\newif\ifHYPER\global\HYPERtrue
\newif\ifJOURNAL\global\JOURNALfalse
\ifJOURNAL
\documentclass[final,1p,times,authoryear]{elsarticle}
\else
\documentclass[10pt,b5paper,fleqn]{article}
\fi
\usepackage[utf8]{inputenc}
\usepackage{amsmath,amsfonts,amssymb}
\usepackage[mathscr]{eucal}
\usepackage{color}
\usepackage{graphicx}
\usepackage{amsthm}

\ifHYPER
\definecolor{ks-green}{rgb}{0.0,0.7,0.0}
\definecolor{ks-red}{rgb}{0.7,0.0,0.0}
\definecolor{ks-blue}{rgb}{0.0,0.0,0.7}
\usepackage{hyperref}
\hypersetup{colorlinks=true,citecolor=ks-green,linkcolor=ks-red}
\fi

\ifJOURNAL
\newtheorem{theorem}{Theorem}
\newtheorem{lemma}[theorem]{Lemma}
\newtheorem{corollary}[theorem]{Corollary}

\newtheorem{proposition}[theorem]{Proposition}
\newtheorem{definition}[theorem]{Definition}
\newtheorem{notation}[theorem]{Notation}
\newtheorem{example}[theorem]{Example}
\newtheorem{remark}[theorem]{Remark}
\newtheorem{Example}[theorem]{Example}
\newtheorem{Remark}[theorem]{Remark}
\newtheorem{algorithm}[theorem]{Algorithm}

\else
\numberwithin{equation}{section}
\theoremstyle{plain}
\newtheorem{theorem}[equation]{Theorem}
\newtheorem{lemma}[equation]{Lemma}
\newtheorem{corollary}[equation]{Corollary}
\newtheorem{proposition}[equation]{Proposition}
\newtheorem{algorithm}[equation]{Algorithm}
\theoremstyle{definition}
\newtheorem{definition}[equation]{Definition}
\newtheorem{Example}[equation]{Example}
\newtheorem{Remark}[equation]{Remark}

\newenvironment{example}{\emph{Example.}}{}
\newenvironment{remark}{\emph{Remark.}}{}
\newenvironment{notation}{\emph{Notation.}}{}
\fi

\newcommand{\mthstrut}{\rule[-0.5ex]{0pt}{2.2ex}}

\newcommand{\lnum}[1]{\makebox[1em][r]{\textnormal{\footnotesize{#1}}}}
\newenvironment{algtest}{\bgroup\bfseries\upshape
  \tabbing
    \hspace*{2em}\=
    \hspace*{1.5em}\=\hspace*{1.5em}\=%
    \hspace*{1.5em}\=\hspace*{1.5em}\= \kill \>}{%
  \endtabbing\egroup}

\newcommand{\block}[1]{\underline{#1}}
\newcommand{\tsfrac}[2]{\textstyle{\frac{#1}{#2}}}

\newcommand{\ldivs}{\!\mid_{\mkern-1mu\text{l}}\!}
\newcommand{\rdivs}{\!\mid_{\mkern-1mu\text{r}}\!}

\newcommand{\atoms}{\mathbf{A}}

\def\trp{^{\!\top}}

\def\inv{^{-1}}

\def\numN{\mathbb{N}}

\def\numQ{\mathbb{Q}}
\def\numR{\mathbb{R}}

\def\revddots{\mathinner{\mkern1mu\raise1pt\vbox{\kern7pt\hbox{.}}\mkern2mu
  \raise4pt\hbox{.}\mkern2mu\raise7pt\hbox{.}\mkern1mu}}

\newcommand{\ncRATS}[2]{#1^{\text{rat}}\langle\!\langle #2\rangle\!\rangle}

\newcommand{\freeALG}[2]{#1\langle #2\rangle}
\newcommand{\freeFLD}[2]{#1(\!\langle #2\rangle\!)}

\DeclareMathOperator{\rank}{rank}

\DeclareMathOperator{\linsp}{span}

\newcommand{\field}[1]{\mathbb{#1}}
\newcommand{\als}[1]{\mathcal{#1}}
\newcommand{\ideal}[1]{\mathfrak{#1}}
\newcommand{\complexity}{\mathcal{O}}

\newcommand{\aclo}[1]{\overline{#1}}
\newcommand{\length}[1]{\vert #1 \vert}

\begin{document}
\ifJOURNAL
\begin{frontmatter}
\fi
\title{On the Factorization of\\
  Non-Commutative Polynomials\\
  (in Free Associative Algebras)}
\ifJOURNAL
\else
\author{Konrad Schrempf%
  \footnote{%
    Contact: math@versibilitas.at,
    Universität Wien, Fakultät für Mathematik,
    Oskar-Morgenstern-Platz~1, 1090 Wien, Austria.
    Supported by the Austrian FWF Project P25510-N26
``Spectra on Lamplighter groups and Free Probability''}
}
\fi

\ifJOURNAL
\else
\maketitle
\fi

\begin{abstract}
We describe a simple approach to factorize non-commutative
polynomials, that is, elements in free associative algebras
(over a \emph{commutative} field),
into atoms (irreducible elements) based on (a special form of)
their minimal linear representations.
To be more specific, a correspondence between factorizations
of an element and upper right blocks of zeros in the system matrix
(of its representation) is established.
The problem is then reduced
to solving a system of polynomial equations (with at most quadratic
terms) with \emph{commuting} unknowns
to compute appropriate transformation matrices (if possible).
\end{abstract}

\ifJOURNAL
\begin{keyword}
\else
{\bfseries Keywords:}
\fi
free associative algebra, factorization of polynomials,
minimal linear representation, companion matrix, free field, non-commutative
formal power series
\ifJOURNAL
\end{keyword}
\end{frontmatter}
\else
\fi

\section*{Introduction}

Considering non-commutative (nc for short) polynomials (elements of the
\emph{free associative algebra}) as elements in its
\emph{universal field of fractions} (free field) 
seems ---at a first glance--- like to take a sledgehammer
to crack a nut. It is maybe more common to view them
as \emph{nc rational series} (or \emph{nc formal power series}).
Therefore one could skip the rather complicated theory behind
free fields, briefly introduced in 
\cite[Section~9.3]{Cohn2003b}
. For details we refer to 
\cite[Section~6.4]{Cohn1995a}
.

However, as we shall see, the \emph{normal form}
of \ifJOURNAL\else Cohn and Reutenauer \fi
\cite{Cohn1994a}
\ provides new insights.
For a given element (in the \emph{free field}) \emph{minimal}
linear representations can be transformed into each other
by invertible matrices over the ground field.
In our case we are interested in the (finite set of)
transformation matrices which yield \emph{all} possible
factorizations (modulo insertion of units and permutation of
commuting factors)
of a given polynomial.

Here we use the concept of \emph{admissible linear systems}
(ALS for short) of
\cite{Cohn1972a}
, closely related to linear representations.
At any step, an ALS (for a nc polynomial)
can be easily transformed into a \emph{proper linear system}
\cite[Section~II.1]{Salomaa1978a}
.

A more general concept for the factorization of arbitrary
elements in free fields (for example non-commutative
rational formal power series) is considered in future work.
There (left and right) divisors will be defined
on the level of linear representations. In this way
additional structure compensates what is missing in fields
(when every non-zero element is invertible).

\medskip
Section~\ref{sec:pf.faa} provides the notation and the basic
(algebraic) setup.
Section~\ref{sec:pf.fnp} contains the main result,
Theorem~\ref{thr:pf.factorization}, which describes a correspondence
between factorizations and upper right blocks of zeros
(in the system matrix). To be able to formulate it,
Proposition~\ref{pro:pf.minmul} (``minimal'' polynomial multiplication)
is necessary. The main idea of its proof can be further developed
to Algorithm~\ref{alg:pf.minals} (minimizing ``polynomial'' admissible
linear systems).
Section~\ref{sec:pf.ncm} generalizes the concept
of companion matrices to provide immediately \emph{minimal}
linear representations for a broader class of elements
in the free associative algebra.
Section~\ref{sec:pf.ex} illustrates the concept step by
step.

\medskip
To our knowledge the literature on the factorization of elements in
free associative algebras is rather sparse.
\ifJOURNAL\else Caruso \fi \cite{Caruso2010a}
\ describes ideas of J.~Davenport using homogenization.
We have not yet investigated how their approach compares to
ours. Some special cases (for example \emph{variable disjoint}
factorizations) are treated in \cite{Arvind2015b}
.

Here we do not consider factorizations in skew polynomial rings
(or rather \emph{domains}), or ---more general--- rings satisfying
the Ore condition
\cite[Section~7.1]{Cohn2003b}
\ or
\cite[Section~0.8]{Cohn1985a}
.
A starting point in this context is
\cite{Heinle2013a}
\ or
\cite{Levandovskyy2018a}
. Factorizations of skew polynomials have various
connections to other areas, see for example 
\cite{Retakh2010a}
\ and 
\cite{Gelfand2001a}
, just to mention two.

\section{Representing Non-Commutative Polynomials}\label{sec:pf.faa}

There are much simpler ways for representing elements
in free associative algebras (in such a way that addition and multiplication
can be defined properly) than the following general presented here.
However, including the inverse we keep the full flexibility,
which could be used later to compute the left (or right)
greatest common divisor of two polynomials $p,q$ by minimizing
the linear representation for $p\inv q$ in which common left factors
would cancel.

\begin{notation}
The set of the natural numbers is denoted by $\numN = \{ 1,2,\ldots \}$,
that including zero by $\numN_0$.
Zero entries in matrices are usually replaced by (lower) dots
to stress the structure of the non-zero entries
unless they result from transformations where there
were possibly non-zero entries before.
We denote by $I_n$ the identity matrix of size $n$
respectively $I$ if the size is clear from the context.
\end{notation}

\smallskip
Let $\field{K}$ be a \emph{commutative} field,
$\aclo{\field{K}}$ its algebraic closure and
$X = \{ x_1, x_2, \ldots, x_d\}$ be a \emph{finite} alphabet.
$\freeALG{\field{K}}{X}$ denotes the \emph{free associative
algebra} or \emph{free $\field{K}$-algebra} 
(or ``algebra of non-commutative polynomials'')
and $\freeFLD{\field{K}}{X}$ denotes the \emph{universal field of
fractions} (or ``free field'') of $\freeALG{\field{K}}{X}$,
see 
\cite{Cohn1995a}
,
\cite{Cohn1999a}
. In our examples the alphabet is usually $X=\{x,y,z\}$.
Including the algebra of nc \emph{rational} series $\ncRATS{\field{K}}{X}$
we have the following chain of inclusions:
$\field{K}\subsetneq \freeALG{\field{K}}{X}
  \subsetneq \ncRATS{\field{K}}{X}
  \subsetneq \freeFLD{\field{K}}{X}$.

The \emph{free monoid}\index{free monoid} $X^*$ generated by $X$
is the set of all
\emph{finite words}\index{finite word}
$x_{i_1} x_{i_2} \cdots x_{i_n}$ with $i_k \in \{ 1,2,\ldots, d \}$.
An element of the alphabet is called \emph{letter}\index{letter},
an element of $X^*$ is called \emph{word}\index{word}.
The multiplication on $X^*$ is just the \emph{concatenation}\index{concatenation}
of words, that is,
$(x_{i_1} \cdots x_{i_m})\cdot (x_{j_1} \cdots x_{j_n})
= x_{i_1} \cdots x_{i_m} x_{j_1} \cdots x_{j_n}$,
with neutral element $1$, the \emph{empty word}.
The \emph{length} of a word $w=x_{i_1} x_{i_2} \cdots x_{i_m}$ is $m$,
denoted by $\length{w} = m$ or $\ell(w) = m$.
For detailled introductions see
\cite[Chapter~1]{Berstel2011a}
\ or
\cite[Section~I.1]{Salomaa1978a}
.

\begin{definition}[Inner Rank, Full Matrix, Hollow Matrix,
see
\cite{Cohn1985a}
, \cite{Cohn1999a}
]\label{def:pf.full}
\index{inner rank}\index{full matrix}\index{hollow matrix}%
Given a matrix $A \in \freeALG{\field{K}}{X}^{n \times n}$, the \emph{inner rank}
of $A$ is the smallest number $m\in \numN$
such that there exists a factorization
$A = T U$ with $T \in \freeALG{\field{K}}{X}^{n \times m}$ and
$U \in \freeALG{\field{K}}{X}^{m \times n}$.
The matrix $A$ is called \emph{full} if $m = n$,
\emph{non-full} otherwise.
It is called \emph{hollow} if it contains a zero submatrix of size
$k \times l$ with $k+l>n$.
\end{definition}

\begin{definition}[Associated Matrices,
\cite{Cohn1995a}
]\label{def:pf.ass}
\index{associated matrix}%
Two matrices $A$ and $B$ over $\freeALG{\field{K}}{X}$ (of the same size)
are called \emph{associated} over a subring $R\subseteq \freeALG{\field{K}}{X}$ 
if there exist (over $R$) invertible matrices $P,Q$ such that
$A = P B Q$.
\end{definition}

\begin{lemma}[%
\protect{\cite[Corollary~6.3.6]{Cohn1995a}
}]\label{lem:pf.cohn95.636}
A linear square matrix over $\freeALG{\field{K}}{X}$
which is not full is associated over $\field{K}$ to a linear
hollow matrix.
\end{lemma}

\begin{remark}
A hollow square matrix cannot be full
\cite[Section~3.2]{Cohn1985a}
, illustrated in an example:
\begin{displaymath}
A =
\begin{bmatrix}
z & . & . \\
x & . & . \\
y & -x & 1 
\end{bmatrix}
=
\begin{bmatrix}
z & 0 \\
x & 0 \\
0 & 1
\end{bmatrix}
\begin{bmatrix}
1 & 0 & 0 \\
y & -x & 1
\end{bmatrix}.
\end{displaymath}
\end{remark}

\begin{definition}[Linear Representations, Dimension, Rank%
\ifJOURNAL
\else
\ \cite{Cohn1994a,Cohn1999a}
\fi]\label{def:pf.rep}
\index{linear representation}%
Let $f \in \freeFLD{\field{K}}{X}$.
A \emph{linear representation} of $f$ is a triple $\pi_f$ = $(u,A,v)$ with
$u \in \field{K}^{1 \times n}$,
full $A = A_0 \otimes 1 + A_1 \otimes x_1 + \ldots
  + A_d \otimes x_d$, $A_\ell \in \field{K}^{n\times n}$ and
$v \in \field{K}^{n\times 1}$ such that $f = u A\inv v$.
The \emph{dimension} of the representation is $\dim \pi_f = n$.
It is called \emph{minimal} if $A$ has the smallest possible dimension
among all linear representations of $f$.
A minimal one $\pi_f$ defines the \emph{rank}\index{rank} of $f$,
$\rank f = \dim \pi_f$.
The ``empty'' representation $\pi = (,,)$ is the minimal one
of $0 \in \freeFLD{\field{K}}{X}$ with $\dim \pi = 0$.
\end{definition}

\begin{definition}[Left and Right Families
\cite{Cohn1994a}
]\label{def:pf.family}
\index{left family}\index{right family}%
Let $\pi=(u,A,v)$ be a linear representation of $f \in \freeFLD{\field{K}}{X}$
of dimension $n$.
The families $( s_1, s_2, \ldots, s_n )\subseteq \freeFLD{\field{K}}{X}$
with $s_i = (A\inv v)_i$
and $( t_1, t_2, \ldots, t_n )\subseteq \freeFLD{\field{K}}{X}$
with $t_j = (u A\inv)_j$
are called \emph{left family} and \emph{right family} respectively.
$L(\pi) = \linsp \{ s_1, s_2, \ldots, s_n \}$ and
$R(\pi) = \linsp \{ t_1, t_2, \ldots, t_n \}$
denote their linear spans (over $\field{K}$).
\end{definition}

\begin{proposition}[%
\cite{Cohn1994a}
, Proposition 4.7]
A representation $\pi=(u,A,v)$ of an element $f \in \freeFLD{\field{K}}{X}$
is minimal if and only if both, the left family
and the right family are $\field{K}$-linearly independent.
\label{pro:pf.cohn94.47}
\end{proposition}

\begin{definition}[Admissible Linear Systems%
\
\cite{Cohn1972a}
. Admissible Transformations]\label{def:pf.als}
\index{admissible linear system}\index{admissible transformation}%
A linear representation $\als{A} = (u,A,v)$ of $f \in \freeFLD{\field{K}}{X}$
is called \emph{admissible linear system} (for $f$),
denoted by $A s = v$,
if $u=e_1=[1,0,\ldots,0]$. The element $f$ is then the first component
of the (unique) solution vector $s$.
Given a linear representation $\als{A} = (u,A,v)$
of dimension $n$ of $f \in \freeFLD{\field{K}}{X}$
and invertible matrices $P,Q \in \field{K}^{n\times n}$,
the transformed $P\als{A}Q = (uQ, PAQ, Pv)$ is
again a linear representation (of $f$).
If $\als{A}$ is an ALS,
the transformation $(P,Q)$ is called
\emph{admissible} if the first row of $Q$ is $e_1 = [1,0,\ldots,0]$.
\end{definition}

Transformations can be done by elementary row- and column operations,
explained in detail in
\cite[Remark~1.12]{Schrempf2017a}
. For further remarks and connections to the related concepts
of linearization and realization see
\cite[Section~1]{Schrempf2017a}
. 
For rational operations on ALS level see the following proposition.
If an ALS is \emph{minimal} then more refined versions of an
inverse give a minimal ALS again. For a detailled discussion
we refer to 
\cite[Section~4]{Schrempf2017a}
.

\begin{proposition}[Rational Operations%
\
\cite{Cohn1999a}
]\label{pro:pf.ratop}
Let $0\neq f,g \in \freeFLD{\field{K}}{X}$ be given by the
admissible linear systems $\als{A}_f = (u_f, A_f, v_f)$
and $\als{A}_g = (u_g, A_g, v_g)$ respectively
and let $0\neq \mu \in \field{K}$.
Then admissible linear systems for the rational operations
can be obtained as follows:

\smallskip\noindent
The scalar multiplication
$\mu f$ is given by
\begin{displaymath}
\mu \als{A}_f =
\bigl( u_f, A_f, \mu v_f \bigr).
\end{displaymath}
The sum $f + g$ is given by
\begin{displaymath}
\als{A}_f + \als{A}_g =
\left(
\begin{bmatrix}
u_f & . 
\end{bmatrix},
\begin{bmatrix}
A_f & -A_f u_f\trp u_g \\
. & A_g
\end{bmatrix}, 
\begin{bmatrix} v_f \\ v_g \end{bmatrix}
\right).
\end{displaymath}
The product $fg$ is given by
\begin{displaymath}
\als{A}_f \cdot \als{A}_g =
\left(
\begin{bmatrix}
u_f & . 
\end{bmatrix},
\begin{bmatrix}
A_f & -v_f u_g \\
. & A_g
\end{bmatrix},
\begin{bmatrix}
. \\ v_g
\end{bmatrix}
\right).
\end{displaymath}
And the inverse $f\inv$ is given by
\begin{displaymath}
\als{A}_f\inv =
\left(
\begin{bmatrix}
1 & . 
\end{bmatrix},
\begin{bmatrix}
-v_f & A_f \\
. & u_f
\end{bmatrix},
\begin{bmatrix}
. \\ 1
\end{bmatrix}
\right).
\end{displaymath}
\end{proposition}

\begin{definition}\label{def:pf.reg}
An element in $\freeFLD{\field{K}}{X}$ is called \emph{regular},
if it has a linear representation $(u,A,v)$ with $A = I - M$,
where $M = M_1 \otimes x_1 + \ldots + M_d \otimes x_d$
with $M_i \in \field{K}^{n \times n}$ for some $n\in \numN$,
that is, $A_0 = I$ in Definition~\ref{def:pf.rep},
or equivalently, if $A_0$ is regular (invertible).
\end{definition}

\begin{remark}
$A = I - M$ is also called a \emph{monic pencil}\index{monic pencil}%
\
\cite{Helton2007b}
. A regular element can also be represented by a
\emph{proper linear system} $s = v + M s$
\cite[Section~II.1]{Salomaa1978a}
.
\end{remark}

\begin{remark}
For a polynomial $p \in \freeALG{\field{K}}{X}
\subseteq \ncRATS{\field{K}}{X} \subseteq \freeFLD{\field{K}}{X}$
the rank of $p$ is just the \emph{Hankel rank}\index{Hankel rank},
that is, the rank of its Hankel matrix\index{Hankel matrix}
$\mathcal{H}(p) = (h_{w_1, w_2})$ ---rows and columns are
indexed by words in the free monoid---
where $h_{w_1, w_2} \in \field{K}$ is the coefficient of the monomial
$w = w_1 w_2$ of $p$. See 
\cite{Fliess1974a}
\ and
\cite[Section~II.3]{Salomaa1978a}
.
\end{remark}

\begin{example}
The Hankel matrix for $p = x(1-yx) = x- xyx$, with
row indices $[1, x, xy, xyx]$ and column indices $[1, x, yx, xyx]$,
is ---without zero rows $\{y,x^2,yx,y^2,$ $x^3,x^2 y, y x^2, \ldots\}$
and columns $\{y,x^2,xy,y^2,x^3,x^2 y, y x^2, \ldots\}$---
\begin{displaymath}
\mathcal{H}(p) =
\begin{bmatrix}
. & 1 & . & -1 \\
1 & . & -1 & . \\
. & -1 & . & . \\
-1 & . & . & . \\
\end{bmatrix}.
\end{displaymath}
Its rank is $4$. Thus $\rank p =4$ and therefore the dimension
of any \emph{minimal} admissible linear system 
is~4, as will be seen later
in the ALS~\eqref{eqn:pf.minmul.1}
in Example~\ref{ex:pf.minmul},
where we show minimality by
Proposition~\ref{pro:pf.cohn94.47}.
\end{example}

\medskip
The following definitions follow mainly 
\cite{Baeth2015a}
\ and are streamlined to our purpose.
We do not need the full generality here.
While there is a rather uniform factorization theory
in the commutative setting
\cite[Section~1.1]{Geroldinger2006a}
, even the ``simplest'' non-commutative case, that is,
a ``unique'' \emph{factorization domain} like the
\emph{free associative algebra}, is not straightforward.
For a general (algebraic) point of view we recommend the survey
\cite{Smertnig2015a}
. Factorization in \emph{free ideal rings} (FIRs)
is discussed in detail in 
\cite[Chapter~3]{Cohn1985a}
. FIRs play an important role in the construction of free fields.
More on ``non-commutative'' factorization can
be found in
\cite{Jordan1989a}
\ and
\cite{Bell2017a}
\ (just to mention a few)
and the literature therein.

\begin{definition}[Similar Right Ideals, Similar Elements
\protect{\cite[Section~3.2]{Cohn1985a}
}]
\index{similar right ideals}%
Let $R$ be a ring. 
Two right ideals $\ideal{a},\ideal{b} \subseteq R$ are called
\emph{similar}, written as $\ideal{a} \sim \ideal{b}$,
if $R/\ideal{a} \cong R/\ideal{b}$ as right $R$-modules.
\index{similar elements}%
Two elements $p,q\in R$ are called \emph{similar} if
their right ideals $pR$ and $qR$ are similar,
that is, $pR \sim qR$.
See also 
\cite[Section~4.1]{Smertnig2015a}
.
\end{definition}

\begin{definition}[Left and Right Coprime Elements
\protect{\cite[Section~2]{Baeth2015a}
}]\label{def:pf.lrcop}
Let $R$ be a domain and $H = R^\bullet = R \setminus \{ 0 \}$.
An element $p \in H$ \emph{left divides}\index{left divides} $q \in H$,
written as $p \ldivs q$, if $q \in pH = \{ ph \mid h \in H \}$.
Two elements $p,q$ are called \emph{left coprime}\index{left coprime}
if for all $h$ such that $h\ldivs p$ and $h\ldivs q$
implies $h \in H^\times = \{ f \in H \mid f \text{ is invertible} \}$,
that is, $h$ is an element of the \emph{group of units}\index{unit group}.
Right division $p \rdivs q$ and the notion of
\emph{right coprime} is defined in a similar way.
Two elements are called \emph{coprime}\index{coprime}
if they are left and right coprime.
\end{definition}

\begin{definition}[Atomic Domains, Irreducible Elements
\protect{\cite[Section~2]{Baeth2015a}
}]\label{def:pf.atoms}
Let $R$ be a domain and $H = R^\bullet$.
An element $p\in H \setminus H^\times$, that is,
a non-zero non-unit (in $R$), is called an \emph{atom}\index{atom}
(or \emph{irreducible}\index{irreducible element})
if $p = q_1 q_2$ with $q_1,q_2 \in H$ implies that
either $q_1 \in H^\times$ or $q_2 \in H^\times$.
The set of atoms in $R$ is denoted by
$\atoms(R)$.
The (cancellative) monoid $H$ is called \emph{atomic}\index{atomic domain}
if every non-unit can be written as a finite product of atoms of $H$.
The domain $R$ is called \emph{atomic} if the monoid $R^\bullet$ is atomic.
\end{definition}

\begin{remark}
Similarity of two elements $a,a'$ in a \emph{weak Bezout ring} $R$
is equivalent to the existence of $b,b' \in R$ such that
$ab' = ba'$ with $ab'$ and $ba'$ coprime,
that is, $a$ and $b$ are left coprime and
$b'$ and $a'$ are right coprime.
The free associative algebra $\freeALG{\field{K}}{X}$ is a weak Bezout ring
\cite[Proposition~5.3 and Theorem~6.2]{Cohn1963b}
.
\end{remark}

\begin{example}
The polynomials $p = (1-xy)$ and $q=(1-yx)$ are similar,
because $p x = (1-xy) x = x - xyx = x(1-yx) = x q$.
See also Example~\ref{ex:pf.minmul}.
\end{example}

\begin{example}
In the \emph{free monoid} $X^*$ the atoms are just the letters
$x_i \in X$.
\end{example}

\begin{definition}[Similarity Unique Factorization Domains
\protect{\cite[Definition~4.1]{Smertnig2015a}
}]\label{def:pf.sfd}
\index{similarity factorization domain}\index{similarity-UFD}%
A domain $R$ is called \emph{similarity factorial}
(or a \emph{similarity-UFD}) if
$R$ is atomic and it satisfies the property that
if $p_1 p_2 \cdots p_m = q_1 q_2 \cdots q_n$ for atoms\index{atom}
(irreducible elements)\index{irreducible element}
$p_i,q_j \in R$,
then $m=n$ and there exists a permutation $\sigma \in \mathfrak{S}_m$
such that $p_i$ is similar to $q_{\sigma(i)}$ for all $i\in 1,2, \ldots, m$.
\end{definition}

\begin{proposition}[%
\protect{\cite[Theorem~6.3]{Cohn1963b}
}]\label{pro:pf.cohn63b}
The free associative algebra $\freeALG{\field{K}}{X}$
is a similarity (unique) factorization domain.
\end{proposition}

\section{Factorizing non-commutative Polynomials}\label{sec:pf.fnp}

Our concept for the factorization of nc polynomials
(Theorem~\ref{thr:pf.factorization})
relies on \emph{minimal} linear representations.
Beside the ``classical'' algorithms from 
\cite{Cardon1980a}
\ and 
\cite{Fornasini1980a}
, there is a naive one
illustrated in Section~\ref{sec:pf.ex}.
The latter is not very efficient in general for an
alphabet with more than one letter
but it preserves the form defined in the following.
After formulating a \emph{minimal polynomial multiplication}
in Proposition~\ref{pro:pf.minmul}
and illustrating it in an example
we present an algorithm which works directly
on the system matrix (of the admissible linear system).

Using \emph{proper linear systems} 
\cite[Section~II.1]{Salomaa1978a}
\ would be slightly too restrictive
because the only possible admissible transformations are
conjugations (with respect to the system matrix).
That these are not sufficient can be seen in
Example~\ref{ex:pf.irred}.
On the other hand, \emph{admissible linear systems} are too general.
Therefore we define a form that is suitable for our purpose.

\begin{remark}
Although we use admissible linear systems,
restricting the application of the rational operations 
(excluding the inverse) in Proposition~\ref{pro:pf.ratop}
to (systems for) polynomials only, one gets again polynomials.
If the inverse is restricted to
(systems for) rational formal power series
with \emph{non-vanishing} constant coefficient
one gets again rational formal power series.
Since we are using multiplication only,
Theorem~\ref{thr:pf.factorization} does \emph{not}
rely on the (construction of the) free field.
\end{remark}

\begin{definition}[Pre-Standard ALS, Pre-Standard Admissible Transformation]%
\label{def:pf.psals}
An admissible linear system $\als{A} = (u,A,v)$
of dimension $n$
with system matrix $A = (a_{ij})$
for a non-zero polynomial $0 \neq p \in \freeALG{\field{K}}{X}$ 
is called
\emph{pre-standard}\index{pre-standard ALS}, if
\begin{itemize}
\item[(1)] $v = [0,\ldots,0,\lambda]\trp$ for some $\lambda \in\field{K}$ and
\item[(2)] $a_{ii}=1$ for $i=1,2,\ldots, n$ and $a_{ij}=0$ for $i>j$,
  that is, $A$ is upper triangular.
\end{itemize}
A pre-standard ALS is also written as $\als{A} = (1,A,\lambda)$
with $1,\lambda \in \field{K}$.
An admissible transformation $(P,Q)$ for an ALS $\als{A}$
is called \emph{pre-standard}\index{pre-standard admissible transformation},
if the transformed system $P\als{A} Q$ is (again) pre-standard.
\end{definition}

\subsection{Minimal Multiplication}\label{sec:pf.fnp.mm}

To be able to prove that the construction
in Proposition~\ref{pro:pf.minmul} leads to a \emph{minimal}
linear representation (for the product of two nc polynomials)
some preparation is necessary.
One of the main tools is Lemma~\ref{lem:pf.cohn95.636}
\cite[Corollary~6.3.6]{Cohn1995a}
.
Although we are working with regular elements only,
invertibility of the constant coefficient matrix $A_0$
(in the system matrix) does not have to be assumed in
Lemma~\ref{lem:pf.rt1}.

\begin{lemma}\label{lem:pf.rt1}
Let $\als{A} = (u,A,v)$ be an ALS
of dimension $n \ge 1$ with $\field{K}$-linearly independent 
left family $s = A\inv v$ and
$B = B_0 \otimes 1 + B_1 \otimes x_1 + \ldots + B_d \otimes x_d$
with $B_\ell\in \field{K}^{m\times n}$, such that
$B s = 0$. Then there exists a (unique) $T \in \field{K}^{m \times n}$
such that $B = TA$.
\end{lemma}

\begin{proof}
The trivial case $n=1$ implies $B=0$ and therefore $T=0$.
Now let $n \ge 2$.
Without loss of generality assume that $v = [0,\ldots,0,1]\trp$
and $m=1$.
Since $A$ is full and thus invertible over the free field,
there exists a unique $T$ such that $B = TA$, namely
$T = B A\inv$ in $\freeFLD{\field{K}}{X}^{1 \times n}$.
The last column in $T$ is zero because
$0 = Bs = BA\inv v = Tv$.
Now let $A'$ denote the matrix $A$ whose last row is removed
and $A'_B$ the matrix obtained from $A$ when the last row is replaced by
$B$. $A'_B$ cannot be full since $s \in \ker A'_B$
would give a contradiction: $s = (A'_B)\inv 0 = 0$.

We claim that there is only one possibility to transform $A'_B$
to a hollow matrix, namely with zero last row. If we cannot produce a
$(n-i) \times i$ block of zeros (by invertible transformations) in the
first $n-1$ rows of $A'_B$, then we cannot get 
blocks of zeros of size $(n-i+1) \times i$ and we are done.

Now assume to the contrary that there are invertible matrices
$P' \in \field{K}^{(n-1) \times (n-1)}$
and (admissible) $Q\in\field{K}^{n \times n}$
with $(Q\inv s)_1 = s_1$,
such that $P'A'Q$ contains a zero block of size $(n-i) \times i$
for some $i=1,\ldots, n-1$. There are two cases. If the first $n-i$
entries in column~1 cannot be made zero, we construct an upper right
zero block:
\begin{displaymath}
\hat{A} = 
\begin{bmatrix}
A_{11} & . \\
A_{21} & A_{22}
\end{bmatrix},
\quad \hat{s} = Q\inv s
\quad\text{and}\quad \hat{v} = P v = v
\end{displaymath}
where $A_{11}$ has size $(n-i) \times (n-i)$.
If $A_{11}$ \emph{were not} full,
then $A$ would not be full (the last row is not involved
in the transformation). Hence this pivot block is
invertible over the free field. 
Therefore $\hat{s}_1 = \hat{s}_2 = \ldots = \hat{s}_{n-i} = 0$.
Otherwise we construct an upper left zero block in $PAQ$.
But then 
$\hat{s}_{i+1} = \hat{s}_{i+2} = \ldots = \hat{s}_{n} = 0$.
Both contradict $\field{K}$-linear independence of the left family.

Hence, by Lemma~\ref{lem:pf.cohn95.636},
$A'_B$ is associated over $\field{K}$
to a linear hollow matrix with a $1 \times n$ block of zeros,
say in the last row (columns are left untouched):
\begin{displaymath}
\begin{bmatrix}
I_{n-1} & . \\
T' & 1
\end{bmatrix}
\begin{bmatrix}
A' \\
B
\end{bmatrix}
I_n
=
\begin{bmatrix}
A' \\
.
\end{bmatrix}.
\end{displaymath}
The matrix $T = [-T', 0] \in \field{K}^{1 \times n}$ satisfies
$B = TA$.
\end{proof}

\begin{remark}
Although the ALS in Lemma~\ref{lem:pf.rt1}
does not have to be minimal,
$\field{K}$-linear independence of the left family
is an important assumption for two reasons.
One is connected to ``pathological'' situations,
compare with 
\cite[Example~2.5]{Schrempf2017a}
.
An entry corresponding
to some $s_j=0$, say for $j=3$, could be arbitrary:
\begin{displaymath}
\begin{bmatrix}
1 & -x & . \\
. & . & z \\
. & 1 & -1
\end{bmatrix}
s
= 
\begin{bmatrix}
. \\ . \\ 1
\end{bmatrix}.
\end{displaymath}
For $B = [2, -2x, y]$ the transformation $T$ has non-scalar entries: $T = [2, y z\inv, 0]$.
The other reason concerns the exclusion of other possibilities for non-fullness
except the last row. For $B= [0, 0, 1]$ the matrix
\begin{displaymath}
A'_B = 
\begin{bmatrix}
1 & -x & . \\
. & . & z \\
. & . & 1
\end{bmatrix}
\end{displaymath}
is hollow. 
However, the
transformation we are looking for here is $T = [0, z\inv, 0]$.
\end{remark}

\begin{lemma}\label{lem:pf.min1}
Let $\als{A} = (u,A,v)$ be
a pre-standard ALS of dimension $n\ge 2$
with $\field{K}$-linearly dependent left family $s=A\inv v$.
Let $A = (a_{ij})$.
Let $m \in \{ 2, 3, \ldots, n \}$ be the minimal index
such that the left subfamily $\underline{s} = (A\inv v)_{i=m}^n$
is $\field{K}$-linearly independent.
Then there exist matrices $T,U \in \field{K}^{1 \times (n+1-m)}$
such that
\begin{displaymath}
U + (a_{m-1,j})_{j=m}^n - T(a_{ij})_{i,j=m}^n  =
\begin{bmatrix}
  0 & \ldots & 0
\end{bmatrix}
\quad\text{and}\quad
T(v_i)_{i=m}^n = 0.
\end{displaymath}
\end{lemma}

\begin{proof}
By assumption, the left subfamily $(s_{m-1}, s_m, \ldots, s_n)$
is $\field{K}$-linearly dependent.
Thus there are $\kappa_j \in \field{K}$ such that
$s_{m-1} = \kappa_m s_m + \kappa_{m+1} s_{m+1} + \ldots + \kappa_n s_n$.
Let $U = [\kappa_m, \kappa_{m+1}, \ldots, \kappa_n ]$.
Then $s_{m-1} - U \underline{s} = 0$.
Since $\als{A}$ is pre-standard, $v_{m-1}=0$.
Now we can apply Lemma~\ref{lem:pf.rt1} with
$B = U + [ a_{m-1,m}, a_{m-1,m+1}, \ldots, a_{m-1,n}]$
(and $\underline{s}$). Hence, there exists
a matrix $T\in \field{K}^{1 \times (n+1-m)}$ such that
\begin{equation}\label{eqn:pf.min1}
U + 
\begin{bmatrix}
a_{m-1,m} & \ldots & a_{m-1,n}
\end{bmatrix}
- T
\begin{bmatrix}
1 & a_{m,m+1} & \ldots & a_{m,n} \\
. & \ddots & \ddots & \vdots \\
. & . & 1 & a_{n-1,n} \\
. & . & . & 1
\end{bmatrix}
=
\begin{bmatrix}
0 & \ldots & 0 
\end{bmatrix}.
\end{equation}
Recall that the last column of $T$ is zero,
whence $T(v_i)_{i=m}^n = 0$.
\end{proof}

\begin{proposition}[Minimal Polynomial Multiplication]\label{pro:pf.minmul}
Let $0\ne p,q \in \freeALG{\field{K}}{X}$ be given by the
\emph{minimal} pre-standard admissible linear systems
$A_p = (u_p, A_p, v_p) = (1, A_p, \lambda_p)$ and
$A_q = (u_q, A_q, v_q) = (1, A_q, \lambda_q)$ of dimension $n_p,n_q \ge 2$ respectively.
Then a \emph{minimal} ALS for $pq$ has dimension $n = n_p + n_q -1$
and can be constructed in the following way:
\begin{itemize}
\item[(1)] Construct the following ALS $\als{A}'=(u',A',v')$ for the product $pq$:
\begin{displaymath}
\begin{bmatrix}
A_p & -v_p u_q \\
. & A_q
\end{bmatrix}
s' =
\begin{bmatrix}
. \\ v_q
\end{bmatrix}.
\end{displaymath}
\item[(2)] Add $\lambda_p$-times column~$n_p$ to column~$n_p+1$
  (the entry $s_{n_p}'$ becomes zero).
\item[(3)] Remove row~$n_p$ of $A'$ and $v'$ and column~$n_p$ of $A'$ and $u'$ and
  denote the new (pre-standard) ALS of dimension $n_p + n_q -1$
  by $\als{A} = (u,A,v) = (1,A,\lambda)$.
\end{itemize}
\end{proposition}
\begin{proof}
The left family of $\als{A}'$ is
\begin{displaymath}
s' = 
\begin{bmatrix}
A_p\inv & A_p\inv v_p u_q \\
. & A_q\inv 
\end{bmatrix}
\begin{bmatrix}
. \\ v_q
\end{bmatrix}
=
\begin{bmatrix}
s_p q \\ s_q
\end{bmatrix}.
\end{displaymath}
Clearly, $\als{A}$ is again pre-standard with $\lambda = \lambda_q$.
Both systems $\als{A}_p$ and $\als{A}_q$ are minimal.
Therefore their left and right families 
are $\field{K}$-linearly independent.
Without loss of generality assume that $\lambda_p = 1$.
Then the last entry $t_{n_p}^p$ of the right family of $\als{A}_p$
is equal to $p$. Let $k = n_p$.
We have to show that both, the left and the right family
\begin{align*}
s &= (s_1, s_2, \ldots, s_n)
   = (s_1^p q,\ldots, s_{k-1}^p q, q, s_2^q, \ldots, s_{n_q}^q), \\
t &= (t_1, t_2, \ldots, t_n)
   = (t_1^p, \ldots, t_{k-1}^p, p, p t_2^q, \ldots, p t_{n_q}^q)
\end{align*}
of $\als{A}$ are $\field{K}$-linearly independent
respectively.
Assume the contrary for $s$, say there
is an index $1< m \le k$ such that $(s_{m-1}, s_m, \ldots, s_n)$
is $\field{K}$-linearly dependent while $(s_m, \ldots, s_n)$
is $\field{K}$-linearly independent.
Then, by Lemma~\ref{lem:pf.min1} there exist
matrices $T,U \in \field{K}^{1 \times (n-m+1)}$
such that $\eqref{eqn:pf.min1}$ holds and therefore
invertible matrices $P,Q \in \field{K}^{n \times n}$,
\begin{displaymath}
P = 
\begin{bmatrix}
I_{m-2} & . & . \\
. & 1 & T \\
. & . & I_{n-m+1}
\end{bmatrix}
\quad\text{and}\quad
Q =
\begin{bmatrix}
I_{m-2} & . & . \\
. & 1 & U \\
. & . & I_{n-m+1}
\end{bmatrix},
\end{displaymath}
that yield equation $s_{m-1} = 0$ (in row~$m-1$) in $P\als{A}Q$.
Let $\tilde{P}$ (respectively $\tilde{Q}$) be the upper
left part of $P$ (respectively $Q$) of size $k \times k$.
Then the equation in row~$m-1$ in $\tilde{P} \als{A}_p \tilde{Q}$
is $s^p_{m-1} = \alpha \in \field{K}$,
contradicting $\field{K}$-linear independence of the left
family of $\als{A}_p$ since $s^p_{k} = \lambda_p \in \field{K}$.
A similar argument (and a variant of Lemma~\ref{lem:pf.min1})
for the right family yields its
$\field{K}$-linear independence.
Hence, by Proposition~\ref{pro:pf.cohn94.47}, the admissible linear
system $\als{A}$ (for $pq$) is minimal.
\end{proof}

\begin{corollary}
Let $0\neq p,q \in \freeALG{\field{K}}{X}$. 
Then $\rank(pq) = \rank(p) + \rank(q) - 1$.
\end{corollary}

\begin{Example}\label{ex:pf.minmul}
The polynomials $p = x \in\freeALG{\field{K}}{X}$
and $q = 1-yx \in\freeALG{\field{K}}{X}$ have the minimal
pre-standard admissible linear systems
\begin{displaymath}
\als{A}_p = \left(
1,
\begin{bmatrix}
1 & - x \\
. & 1 
\end{bmatrix},
1
\right)
\quad\text{and}\quad
\als{A}_q = \left(
1,
\begin{bmatrix}
1 & y & -1 \\
. & 1 & -x \\
. & . & 1
\end{bmatrix},
1
\right)
\end{displaymath}
respectively. Then a pre-standard ALS for $pq = x(1-yx)$ is given by
\begin{displaymath}
\begin{bmatrix}
1 & -x & . & . & . \\
. & 1 & -1 & . & . \\
. & . & 1 & y & -1 \\
. & . & . & 1 & -x \\
. & . & . & . & 1 
\end{bmatrix}
s =
\begin{bmatrix}
. \\ . \\ . \\ . \\ 1
\end{bmatrix},
\quad
s =
\begin{bmatrix}
x(1-yx) \\
1-yx \\
1-yx \\
x \\
1
\end{bmatrix}.
\end{displaymath}
Adding column~2 to column~3 (and subtracting $s_3$ from $s_2$)
yields
\begin{displaymath}
\begin{bmatrix}
1 & -x & -x & . & . \\
. & 1 & 0 & . & . \\
. & . & 1 & y & -1 \\
. & . & . & 1 & -x \\
. & . & . & . & 1 
\end{bmatrix}
s =
\begin{bmatrix}
. \\ . \\ . \\ . \\ 1
\end{bmatrix},
\quad
s =
\begin{bmatrix}
x(1-yx) \\
0 \\
1-yx \\
x \\
1
\end{bmatrix},
\end{displaymath}
thus the pre-standard ALS
\begin{equation}\label{eqn:pf.minmul.1}
\als{A} = \left(
\begin{bmatrix}
1 & . & . & .
\end{bmatrix},
\begin{bmatrix}
1 & -x & 0 & 0 \\
. & 1 & y & -1 \\
. & . & 1 & -x \\
. & . & . & 1 
\end{bmatrix},
\begin{bmatrix}
. \\ . \\ . \\ 1
\end{bmatrix}
\right).
\end{equation}
Since also the right family $t = [1, x, -xy, x(1-yx)]$ is
$\field{K}$-linearly independent, this system is
\emph{minimal} by Proposition~\ref{pro:pf.cohn94.47}.
Note the upper right $1 \times 2$ block of zeros in the system matrix
of $\als{A}$.
\end{Example}

\subsection{Minimizing a pre-standard ALS}\label{sec:pf.fnp.min}

A close look on the proof of Proposition~\ref{pro:pf.minmul}
reveals a surprisingly simple algorithm
for the construction of a \emph{minimal}
pre-standard admissible linear system, provided that one
in pre-standard form is given. It can be used for minimizing
the ALS for the sum in Proposition~\ref{pro:pf.ratop}.
``Simple'' means that it can easily be done manually
also for quite large sparse systems. Depending on the data structure,
the implementation itself is somewhat technical.
One has to be very careful if the scalars (from the ground
field $\field{K}$) cannot be represented exactly,
especially when systems of linear equations (see below)
have to be solved.

To illustrate the main idea we (partially) minimize
a \emph{non-minimal} ``almost'' pre-standard
ALS $\als{A} = (u,A,v)$ of dimension $n=6$
for $p = -xy + (xy + z)$. Note that we do not need knowledge
of the left and right family at all.
Let
\begin{equation}\label{eqn:pf.min0}
\als{A} = \left(
\begin{bmatrix}
1 & . & . & . & . & . 
\end{bmatrix},
\begin{bmatrix}
1 & -x & . & -1 & . & . \\
. & 1 & -y & . & . & . \\
. & . &  1 & . & . & . \\
. & . & . & 1 & -x & -z \\
. & . & . & . & 1 & -y \\
. & . & . & . & . & 1 
\end{bmatrix},
\begin{bmatrix}
. \\ . \\ -1 \\ . \\ . \\ 1
\end{bmatrix}
\right).
\end{equation}
First we do one ``left'' minimization step,
that is, we remove (if possible)
one element of the $\field{K}$-linearly
dependent left family $s = A\inv v$
and construct a new system.
We fix a $1 \le k < n$, say $k=3$.
If we find a (pre-standard admissible)
transformation $(P,Q)$ of the form
\begin{equation}\label{eqn:pf.ltrf}
(P,Q) = \left(
\begin{bmatrix}
I_{k-1} & . & . \\
. & 1 & T \\
. & . & I_{n-k}
\end{bmatrix},
\begin{bmatrix}
I_{k-1} & . & . \\
. & 1 & U \\
. & . & I_{n-k}
\end{bmatrix}
\right)
\end{equation}
such that row~$k$ in $PAQ$ is $[0,0,1,0,0,0]$
and $(Pv)_k = 0$, we can eliminate row~$k$ and
column~$k$ in $P \als{A} Q$ because $(Q\inv s)_k = 0$.
How can we find these blocks $T,U \in \field{K}^{1 \times (n-k)}$?
We write $\als{A}$ in block form
---block row and column indices are
underlined to distinguish them from component indices---
(with respect to row/column~$k$) 
\begin{equation}\label{eqn:pf.bloals}
\als{A}^{[k]} = \left(
\begin{bmatrix}
u_{\block{1}} & . & . 
\end{bmatrix},
\begin{bmatrix}
A_{1,1} & A_{1,2} & A_{1,3} \\
. & 1 & A_{2,3} \\
. & . & A_{3,3}
\end{bmatrix},
\begin{bmatrix}
v_{\block{1}} \\ v_{\block{2}} \\ v_{\block{3}}
\end{bmatrix}
\right)
\end{equation}
and apply the transformation $(P,Q)$:
\begin{align*}
P A Q &= 
\begin{bmatrix}
I_{k-1} & . & . \\
. & 1 & T \\
. & . & I_{n-k}
\end{bmatrix}
\begin{bmatrix}
A_{1,1} & A_{1,2} & A_{1,3} \\
. & 1 & A_{2,3} \\
. & . & A_{3,3}
\end{bmatrix}
\begin{bmatrix}
I_{k-1} & . & . \\
. & 1 & U \\
. & . & I_{n-k}
\end{bmatrix} \\
&=
\begin{bmatrix}
A_{1,1} & A_{1,2} & A_{1,2} U + A_{1,3} \\
. & 1 & U + A_{2,3} + T A_{3,3} \\
. & . & A_{3,3}
\end{bmatrix}, \\
Pv &= 
\begin{bmatrix}
I_{k-1} & . & . \\
. & 1 & T \\
. & . & I_{n-k}
\end{bmatrix}
\begin{bmatrix}
v_{\block{1}} \\ v_{\block{2}} \\ v_{\block{3}}
\end{bmatrix}
=
\begin{bmatrix}
v_{\block{1}} \\ v_{\block{2}} + T v_{\block{3}} \\ v_{\block{3}}
\end{bmatrix}.
\end{align*}
Now we can read of a \emph{sufficient} condition
for $(Q\inv s)_k = 0$, namely
the \emph{existence} of $T,U \in \field{K}^{1 \times (n-k)}$ such that
\begin{equation}\label{eqn:pf.lmsys}
U + A_{2,3} + T A_{3,3} = 0
\quad\text{and}\quad
v_{\block{2}} + T v_{\block{3}} = 0.
\end{equation}
Let $d$ be the number of letters in our alphabet $X$.
The blocks $T= [\alpha_{k+1}, \alpha_{k+2}, \ldots, \alpha_n]$ and
$U = [\beta_{k+1}, \beta_{k+2}, \ldots, \beta_n]$ in the transformation $(P,Q)$
are of size $1 \times (n-k)$, thus we have
a \emph{linear} system of equations (over $\field{K}$) with
$2(n-k)$ unknowns (for $k>1$) and $(d+1)(n-k) + 1$ equations:
\begin{align*}
\begin{bmatrix}
\beta_{k+1} & \beta_{k+2} & \beta_{k+3}
\end{bmatrix}
+
\begin{bmatrix}
0 & 0 & 0 
\end{bmatrix}
+
\begin{bmatrix}
\alpha_{k+1} & \alpha_{k+2} & \alpha_{k+3}
\end{bmatrix}
\begin{bmatrix}
1 & -x & -z \\
. & 1 & -y \\
. & . & 1
\end{bmatrix}
&=
\begin{bmatrix}
0 & 0 & 0
\end{bmatrix}, \\
\begin{bmatrix}
-1
\end{bmatrix}
+
\begin{bmatrix}
\alpha_{k+1} & \alpha_{k+2} & \alpha_{k+3}
\end{bmatrix}
\begin{bmatrix}
. \\ . \\ 1
\end{bmatrix}
&=
\begin{bmatrix}
0
\end{bmatrix}.
\end{align*}
One solution is $T=[0,0,1]$ and $U=[0,0,-1]$.
We compute $\tilde{\als{A}}_1 = P\als{A}Q$ 
and remove block row~$\block{2}$
and column~$\block{2}$,
that is, row~$k$ and column~$k$, to get the new ALS
\begin{displaymath}
\als{A}_1 = (u,A,v) = \left(
\begin{bmatrix}
1 & . & . & . & .
\end{bmatrix},
\begin{bmatrix}
1 & -x & -1 & . & . \\
. & 1 & . & . & y \\
. & . & 1 & -x & -z \\
. & . & . & 1 & -y \\
. & . & . & . & 1
\end{bmatrix},
\begin{bmatrix}
. \\ . \\ . \\ . \\ 1
\end{bmatrix}
\right).
\end{displaymath}
Next we do one ``right'' minimization step, that is, we remove (if possible)
one element of the $\field{K}$-linearly dependent right family $t = u A\inv$
and construct a new system.
We fix a $1 < k \le n=5$, say $k=3$. Note that $t_1 = 1$.
Now we are looking for a transformation $(P,Q)$ of the form
\begin{equation}\label{eqn:pf.rtrf}
(P,Q) = \left(
\begin{bmatrix}
I_{k-1} & T & . \\
. & 1 & . \\
. & . & I_{n-k}
\end{bmatrix},
\begin{bmatrix}
I_{k-1} & U & . \\
. & 1 & . \\
. & . & I_{n-k}
\end{bmatrix}
\right)
\end{equation}
such that column~$k$ in $PAQ$ is $[0,0,1,0,0]\trp$ (for an
admissible transformation, that is,
$U$ has a zero first row, the corresponding
entry $u_k$ in $u$ is zero).
A sufficient
condition for $(t P\inv)_k=0$ is the existence of
$T,U \in \field{K}^{(k-1) \times 1}$ such that
\begin{equation}\label{eqn:pf.rmsys}
A_{1,1} U + A_{1,2} + T = 0.
\end{equation}
(In Remark~\ref{rem:pf.complexity}
there is a less ``compressed'' version of this linear system
of equations.)
One solution is $T=[1,0]\trp$ and $U=[0,0]\trp$.
We compute $\tilde{\als{A}}_2 = P \als{A}_1 Q$ 
and remove
row~$k$ and column~$k$ to get the new (not yet minimal) ALS
\begin{displaymath}
\als{A}_2 = (u,A,v) = \left(
\begin{bmatrix}
1 & . & . & . 
\end{bmatrix},
\begin{bmatrix}
1 & -x & -x & -z \\
. & 1 & . & y \\
. & . & 1 & -y \\
. & . & . & 1
\end{bmatrix},
\begin{bmatrix}
. \\ . \\ . \\ 1
\end{bmatrix}
\right).
\end{displaymath}
If a left (respectively right) minimization step with $k=1$
(respectively $k=n$ and $v = [0,\ldots,0,\lambda]\trp$)
can be done, then the ALS represents zero and we can stop immediately.

The following is the only non-trival observation:
Recall that, if there exist row (respectively column) blocks $T,U$ such that
\eqref{eqn:pf.lmsys} (respectively \eqref{eqn:pf.rmsys})
has a solution then the left (respectively right) family
is $\field{K}$-linearly dependent.
To guarantee \emph{minimality} by Proposition~\ref{pro:pf.cohn94.47}
we need the other implication, that is,
the existence of appropriate row \emph{or} column blocks
for non-minimal pre-standard admissible linear systems.

Although the arguments can be found in the proof of
Proposition~\ref{pro:pf.minmul}, we repeat them here
because this is the crucial part of the minimization
algorithm:
Let $\als{A} = (u,A,v)$ be a pre-standard ALS of
dimension $n \ge 2$ with left family $s = (s_1, s_2, \ldots, s_n)$
and assume that there exists a $1 \le k < n$ such that
the subfamily $(s_{k+1}, s_{k+2}, \ldots, s_n)$ is $\field{K}$-linearly
independent while $(s_k, s_{k+1}, \ldots, s_n)$ is
$\field{K}$-linearly dependent.
Then, by Lemma~\ref{lem:pf.min1}, there exist
matrices $T,U \in \field{K}^{1 \times (n-k)}$
such that \eqref{eqn:pf.lmsys} holds.
In other words: We have to start with $k_s=n-1$ for a left
and $k_t=2$ for a right minimization step.

If we apply one minimization step, we must check
the other family ``again'', illustrated in the following
example, which is \emph{not} constructed out of two
minimal systems:
\begin{displaymath}
\als{A} = (u,A,v) = \left(
\begin{bmatrix}
1 & . & . & . & . 
\end{bmatrix},
\begin{bmatrix}
1 & -x & -y & x+y & . \\
. & 1 & . & . & -z \\
. & . & 1 & . & -z \\
. & . & . & 1 & -y \\
. & . & . & . & 1 \\
\end{bmatrix},
\begin{bmatrix}
. \\ . \\ . \\ . \\ 1
\end{bmatrix}
\right).
\end{displaymath}
Clearly, the left subfamily
$(s_3,s_4,s_5)$ 
and the right subfamily $(t_1, t_2, t_3)$
of $\als{A}$ 
are $\field{K}$-linearly independent respectively.
If we subtract row~3 from row~2 and add column~2 to column~3,
we get the ALS
\begin{displaymath}
\als{A}' = (u',A',v') = \left(
\begin{bmatrix}
1 & . & . & . & . 
\end{bmatrix},
\begin{bmatrix}
1 & -x & -x-y & x+y & . \\
. & 1 & 0 & . & 0 \\
. & . & 1 & . & -z \\
. & . & . & 1 & -y \\
. & . & . & . & 1 \\
\end{bmatrix},
\begin{bmatrix}
. \\ . \\ . \\ . \\ 1
\end{bmatrix}
\right).
\end{displaymath}
The right subfamily $(t_1'',t_2'',t_3'')$
of $\als{A}'' = \als{A}'\mthstrut^{[-2]}$
is (here) \emph{not} $\field{K}$-linearly independent anymore,
therefore we must check for a right minimization step
for $k=3$ again.

\begin{definition}[Minimization Equations, Minimization Transformations]
\label{def:pf.meqn}
Let $\als{A} = (u,A,v)$ be a pre-standard ALS of dimension
$n \ge 2$.
Recall the block decomposition \eqref{eqn:pf.bloals}
\begin{displaymath}
\als{A}^{[k]} = \left(
\begin{bmatrix}
u_{\block{1}} & . & . 
\end{bmatrix},
\begin{bmatrix}
A_{1,1} & A_{1,2} & A_{1,3} \\
. & 1 & A_{2,3} \\
. & . & A_{3,3}
\end{bmatrix},
\begin{bmatrix}
v_{\block{1}} \\ v_{\block{2}} \\ v_{\block{3}}
\end{bmatrix}
\right).
\end{displaymath}
By $\als{A}^{[-k]}$ we denote the ALS $\als{A}^{[k]}$
without row/column~$k$
(of dimension $n-1$):
\begin{displaymath}
\als{A}^{[-k]} = \left(
\begin{bmatrix}
u_{\block{1}} & . 
\end{bmatrix},
\begin{bmatrix}
A_{1,1} & A_{1,3} \\
. & A_{3,3}
\end{bmatrix},
\begin{bmatrix}
v_{\block{1}} \\ v_{\block{3}}
\end{bmatrix}
\right).
\end{displaymath}
For $k = \{ 1,2,\ldots, n-1 \}$ the equations 
$U + A_{2,3} + T A_{3,3} = 0$ and $v_{\block{2}} + T v_{\block{3}} = 0$,
see \eqref{eqn:pf.lmsys},
with respect to the block decomposition $\als{A}^{[k]}$
are called \emph{left minimization equations}%
\index{left minimization equations},
denoted by $\mathcal{L}_k = \mathcal{L}_k(\als{A})$.
A solution by the row block pair $(T,U)$ is denoted by
$\mathcal{L}_k(T,U) = 0$, 
the corresponding transformation $(P,Q) = \bigl(P(T), Q(U) \bigr)$,
see \eqref{eqn:pf.ltrf},
is called \emph{left minimization transformation}%
\index{left minimization transformation}.
For $k = \{ 2,3,\ldots, n \}$ the equations
$A_{1,1} U + A_{1,2} + T = 0$, see \eqref{eqn:pf.rmsys},
with respect to the block decomposition $\als{A}^{[k]}$
are called \emph{right minimization equations}%
\index{right minimization equations},
denoted by $\mathcal{R}_k = \mathcal{R}_k(\als{A})$.
A solution by the column block pair $(T,U)$ is denoted by
$\mathcal{R}_k(T,U) = 0$,
the corresponding transformation $(P,Q) = \bigl(P(T), Q(U) \bigr)$,
see \eqref{eqn:pf.rtrf},
is called \emph{right minimization transformation}%
\index{right minimization transformation}.
\end{definition}

Now there is only one important detail left, namely
that we cannot apply Lemma~\ref{lem:pf.min1} in the following
(first) left minimization step. For $0 \neq \alpha \in \field{K}$
consider the ALS $\als{A}$
\begin{displaymath}
\begin{bmatrix}
1 & -\alpha \\
. & 1 
\end{bmatrix}
s =
\begin{bmatrix}
. \\ \lambda
\end{bmatrix}
\end{displaymath}
and (admissibly) transform the system matrix $A$ in the following way:
\begin{displaymath}
\underbrace{%
\begin{bmatrix}
. & \alpha \\
1 & . 
\end{bmatrix}}_{=:P}
\begin{bmatrix}
1 & -\alpha \\
. & 1 
\end{bmatrix}
\underbrace{%
\begin{bmatrix}
1 & . \\
1/\alpha & 1
\end{bmatrix}}_{=:Q}
=
\begin{bmatrix}
. & \alpha \\
1 & -\alpha
\end{bmatrix}
\begin{bmatrix}
1 & . \\
1/\alpha & 1
\end{bmatrix} \\
=
\begin{bmatrix}
1 & \alpha \\
0 & -\alpha
\end{bmatrix}
\end{displaymath}
thus
\begin{equation}\label{eqn:pf.left1}
\begin{bmatrix}
1 & \alpha \\
. & -\alpha
\end{bmatrix}
s = 
\begin{bmatrix}
\alpha \lambda \\
0
\end{bmatrix}
\end{equation}
and we can remove the \emph{last} row and column.
Note that we do not have to consider such a special case
for the right family.

\begin{algorithm}[Minimizing a pre-standard ALS]\label{alg:pf.minals}
\ \\
Input: $\als{A} = (u,A,v)$ pre-standard ALS
  of dimension $n \ge 2$ (for some polynomial $p$).\\
Output: $\als{A}' = (,,)$ if $p=0$ or
        a minimal pre-standard ALS $\als{A}' = (u',A',v')$ if $p \neq 0$.
        
\begin{algtest}
\hbox{}\\[-3ex]
\lnum{1:}\>$k := 2$ \\
\lnum{2:}\>while $k \le \dim \als{A}$ do \\
\lnum{3:}\>\>$n := \dim(\als{A})$ \\
\lnum{4:}\>\>$k' := n  +1 - k$ \\
\lnum{  }\>\>\textnormal{Is the left subfamily
  \raisebox{0pt}[0pt][0pt]{%
    $(s_{k'},\overbrace{\mthstrut s_{k'+1},\ldots, s_{n}}^{\text{lin.~indep.}})$}
    $\field{K}$-linearly dependent?} \\
\lnum{5:}\>\>if $\exists\, T,U \in \field{K}^{1 \times (k-1)}
  \textnormal{ admissible}: \mathcal{L}_{k'}(\als{A})=\mathcal{L}_{k'}(T,U)=0$ then \\
\lnum{6:}\>\>\>if $k' = 1$ then \\
\lnum{7:}\>\>\>\>return $(,,)$ \\
\lnum{  }\>\>\>endif \\
\lnum{8:}\>\>\>\raisebox{0pt}[0pt][0pt]{%
      $\als{A} := \bigl(P(T) \als{A} Q(U)\bigr)\mthstrut^{[-k']}$} \\
\lnum{9:}\>\>\>if $k > \max \bigl\{ 2, \frac{n+1}{2} \bigr\}$ then \\
\lnum{10:}\>\>\>\>$k := k-1$ \\
\lnum{   }\>\>\>endif \\
\lnum{11:}\>\>\>continue \\
\lnum{   }\>\>endif \\
\lnum{12:}\>\>if $k = 2$ and $s_{n-1} = \alpha s_n$ \textnormal{(for some $\alpha\in \field{K}$)} then \\
\lnum{13:}\>\>\>find \textnormal{admissible} $(P,Q)$ such that $(Q\inv s)_n=0$ 
  and $PAQ$ is \textnormal{upper triangular} \\
\lnum{14:}\>\>\>$\als{A} := (P\als{A} Q)^{[-n]}$ \\
\lnum{15:}\>\>\>continue \\
\lnum{   }\>\>endif \\
\lnum{   }\>\>\textnormal{Is the right subfamily
      \raisebox{0pt}[0pt][0pt]{%
            $(\overbrace{\mthstrut t_1, \ldots,t_{k-1}}^{\text{lin.~indep.}}, t_k)$}
          $\field{K}$-linearly dependent?} \\
\lnum{16:}\>\>if $\exists\, T,U \in \field{K}^{(k-1) \times 1}
    \textnormal{ admissible} : 
     \mathcal{R}_k(\als{A}) = \mathcal{R}_k(T,U)=0$ then \\
\lnum{17:}\>\>\>$\als{A} := \bigl(P(T) \als{A} Q(U) \bigr)\mthstrut^{[-k]}$ \\
\lnum{18:}\>\>\>if $k > \max \bigl\{ 2, \frac{n+1}{2} \bigr\}$ then \\
\lnum{19:}\>\>\>\>$k := k-1$ \\
\lnum{   }\>\>\>endif \\
\lnum{20:}\>\>\>continue \\
\lnum{   }\>\>endif \\
\lnum{21:}\>\>$k := k+1$ \\
\lnum{   }\>done \\
\lnum{22:}\>return $P\als{A},$ 
        \textnormal{with $P$, such that $Pv = [0,\ldots,0,\lambda]\trp$}
\end{algtest}
\end{algorithm}

\begin{proof}
The admissible linear system $\als{A}$
represents $p=0$ if and only if $s_1 = (A\inv v)_1 = 0$.
Since all systems are equivalent to $\als{A}$,
this case is recognized for $k'=1$ because
by Lemma~\ref{lem:pf.min1} there is an \emph{admissible}
transformation such that the first left minimization equation
is fulfilled.
Now assume $p \neq 0$.
We have to show that both, the left family $s'$ and
the right family $t'$ of $\als{A}' = (u',A',v')$ are
$\field{K}$-linearly independent respectively.
Let $n' = \dim(\als{A}')$ and for $k \in \{ 1, 2, \ldots, n' \}$
denote by
\raisebox{0pt}[0pt][0pt]{%
$s'_{(k)} = (s'_{n'+1-k}, s'_{n'+2-k}, \ldots, s'_{n'})$}
the left and by 
\raisebox{0pt}[0pt][0pt]{%
$t'_{(k)} = (t'_1, t'_2, \ldots, t'_k)$}
the right subfamily.
By assumption $\als{A}$ is a pre-standard ALS and
therefore $s'_{n'} \neq 0$ and $t'_1 \neq 0$,
that is, $s'_{(1)}$ is $\field{K}$-linearly independent
and $t'_{(1)}$ is $\field{K}$-linearly independent.
The loop starts with $k=2$.
Only if both, $s'_{(k)}$ and $t'_{(k)}$
are $\field{K}$-linearly indpendent respectively,
$k$ is incremented. For $k=2$ we can solve the special
case by~\eqref{eqn:pf.left1} for the left family.
Otherwise a left (Lemma~\ref{lem:pf.min1}) or a right
(variant of Lemma~\ref{lem:pf.min1}) minimization step
was successful and the dimension of the current ALS is
strictly smaller than that of the previous.
Hence, since $k$ is bounded from below,
the algorithm stops in a finite number of steps.
We just have to make sure that there exists an
admissible transformation, if there exists a
column couple $(T,U)$ such that the right minimization
equations $\mathcal{R}_k(T,U)=0$ are fulfilled.
However, if the first column would be involved to
eliminate the first entry in column~$k$, the
$k$-th row could be used instead.
Clearly, $\als{A}'$ is in pre-standard form.
\end{proof}

\begin{Remark}\label{rem:pf.complexity}
This algorithm can be implemented very efficiently
provided that row and column transformations are
done directly. Let $d$ be the number of letters in our alphabet $X$.
For $\ell=0,1,\ldots,d$ let $A_{ij}^{(\ell)}$ denote the
submatrix corresponding to letter $x_\ell$ and the (current)
block decomposition of $\als{A}^{[k]}$.
The right minimization equations
$A_{1,1} U + A_{1,2} + T = 0$ can be written as
\begin{displaymath}
\begin{bmatrix}
I & A_{1,1}^{(0)} \\
I & A_{1,1}^{(1)} \\
\vdots & \vdots \\
I & A_{1,1}^{(d)}
\end{bmatrix}
\begin{bmatrix}
T \\ U
\end{bmatrix}
=
\begin{bmatrix}
- A_{1,2}^{(0)} \\
- A_{1,2}^{(1)} \\
\vdots \\
- A_{1,2}^{(d)}
\end{bmatrix}
\end{displaymath}
with $2(k-1)$ unknowns, $k<n$.
By Gaussian elimination one
gets complexity $\complexity(dn^3)$
for solving such a system,
see \cite[Section~2.3]{Demmel1997a}
.
To build such a system and working on a
linear matrix pencil
$\bigl[\begin{smallmatrix} 0 & u \\ v & A \end{smallmatrix}\bigr]$
with $d+1$ square
coefficient matrices of size $n+1$
(transformations, etc.)
has complexity $\complexity(dn^2)$.
Since there are at most $2(n-1)$ steps,
we get overall (minimization) complexity
$\complexity(d n^4)$.
The algorithm of \cite{Cardon1980a}
\ has complexity $\complexity(d n^3)$.
One has to be careful with a direct comparison.
The latter works more general for regular elements,
that is, rational formal power series.
However the idea used here generalizes directly for
larger blocks, say for a block decomposition $\als{A}^{[k,k+1,\ldots,k+m]}$
for $k,m < n$ and can be used partially for \emph{non-regular} elements
in the free field, for example to solve the \emph{word problem}
which would have complexity $\complexity(d n^6)$.
Further details can be found in 
\cite{Schrempf2017a}
. 
\end{Remark}

\subsection{Factorizing non-commutative Polynomials}\label{sec:pf.fnp.fact}

\begin{definition}[Atomic Admissible Linear Systems]\label{def:pf.atom}
A \emph{minimal} pre-standard ALS $\als{A} = (1,A,\lambda)$
of dimension $n\ge 2$ is called \emph{atomic}\index{atomic ALS}
(\emph{irreducible})\index{irreducible ALS},
if there is no pre-standard admissible transformation $(P,Q)$ such that
$PAQ$ has an upper right block of zeros of size $(n-i-1) \times i$ for some
$i=1,2,\ldots,n-2$.
\end{definition}

If we take the minimal ALS \eqref{eqn:pf.minmul.1}
for $pq = x(1-yx)$ 
from Example~\ref{ex:pf.minmul} and
add column~2 to column~4, we obtain
\begin{displaymath}
\begin{bmatrix}
1 & -x & . & -x \\
. & 1 & y & 0 \\
. & . & 1 & -x \\
. & . & . & 1 
\end{bmatrix}
s =
\begin{bmatrix}
. \\ . \\ . \\ 1
\end{bmatrix},
\quad
s =
\begin{bmatrix}
x(1-yx) \\
-yx \\
x \\
1
\end{bmatrix}.
\end{displaymath}
Subtracting row~3 from row~1 yields
\begin{displaymath}
\begin{bmatrix}
1 & -x & -1 & 0 \\
. & 1 & y & 0 \\
. & . & 1 & -x \\
. & . & . & 1 
\end{bmatrix}
s =
\begin{bmatrix}
. \\ . \\ . \\ 1
\end{bmatrix},
\quad
s =
\begin{bmatrix}
x(1-yx) \\
-yx \\
x \\
1
\end{bmatrix}
\end{displaymath}
with an upper right $2 \times 1$ block of zeros in the system matrix.
We would obtain the same system by \emph{minimal polynomial multiplication}
of $1-xy$ and $x$. This illustrates that we can find the factors of
non-commutative polynomials by searching for
pre-standard admissible transformations
that give a corresponding upper right block of zeros in the transformed
system matrix.

Theorem~\ref{thr:pf.factorization}
will establish the correspondence between a factorization into
atoms and the structure of the upper right blocks of zeros.
Thus, to factorize a polynomial $p$ of rank~$n\ge 3$ into \emph{non-trivial} factors
$p = q_1 q_2$ with $\rank(q_i)= n_i \ge 2$ and $n = n_1 + n_2 - 1$,
we have to look for (pre-standard admissible) transformations of the form
\begin{equation}\label{eqn:pf.factrn}
(P,Q) = \left(
\begin{bmatrix}
1 & \alpha_{1,2} & \ldots & \alpha_{1,n-1} & 0 \\
  & \ddots & \ddots & \vdots & \vdots \\
  &   & 1 & \alpha_{n-2,n-1} & 0 \\
  &   &   & 1 & 0 \\
  &   &   &   & 1
\end{bmatrix},
\begin{bmatrix}
1 & 0 & 0 & \ldots & 0 \\
  & 1 & \beta_{2,3} & \ldots & \beta_{2,n} \\
  &   & 1 & \ddots & \vdots  \\
  &   &   & \ddots & \beta_{n-1,n} \\
  &   &   &   & 1 \\
\end{bmatrix}
\right)
\end{equation}
with entries $\alpha_{ij}, \beta_{ij} \in \field{K}$.
In general this is a \emph{non-linear} problem
with $(n-2)(n-1)$ unknowns.

\begin{definition}[Standard Admissible Linear Systems]\label{def:pf.stdals}
Every
\begin{itemize}
\item ALS $(1, I_1, \lambda)$ for a scalar $0 \neq \lambda \in \field{K}$,
\item atomic ALS and
\item non-atomic \emph{minimal} pre-standard ALS
  $\als{A} = (1,A,\lambda)$ for $p \in \freeALG{\field{K}}{X}$
  of dimension $n \ge 3$ obtained
  from minimal multiplication (Proposition~\ref{pro:pf.minmul}) 
  of atomic admissible linear systems for
  its atomic factors $p = q_1 q_2 \cdots q_m$ (with atoms $q_i$)
\end{itemize}
is called \emph{standard}\index{standard ALS}.
\end{definition}

\begin{remark}
We have the following chain of ``inclusions'' of admissible linear systems:
Atomic ALS $\subsetneq$ Standard ALS $\subsetneq$ Pre-Standard ALS
$\subsetneq$ ALS.
\end{remark}

\begin{Example}\label{ex:pf.irred}
Let $p = x^2 -2\in \freeALG{\field{K}}{X}$ be given by the minimal ALS
\begin{displaymath}
\als{A} = \left(
\begin{bmatrix}
1 & . & . 
\end{bmatrix},
\begin{bmatrix}
1 & -x & 2 \\
. & 1 & -x \\
. & . & 1
\end{bmatrix},
\begin{bmatrix}
. \\ . \\ 1
\end{bmatrix}
\right).
\end{displaymath}
If $\field{K} = \numQ$ then $\als{A}$ is atomic (irreducible).
If $\field{K} = \numR$ then there is the pre-standard admissible transformation
\begin{displaymath}
(P,Q) = \left(
\begin{bmatrix}
1 & \sqrt{2} & . \\
. & 1 & . \\
. & . & 1
\end{bmatrix},
\begin{bmatrix}
1 & . & . \\
. & 1 & -\sqrt{2} \\
. & . & 1
\end{bmatrix}
\right)
\end{displaymath}
such that $\als{A}' =P\als{A} Q$ is
\begin{displaymath}
\als{A}'= \left(
\begin{bmatrix}
1 & . & . 
\end{bmatrix},
\begin{bmatrix}
1 & -x+\sqrt{2} & 0 \\
. & 1 & -x-\sqrt{2} \\
. & . & 1
\end{bmatrix},
\begin{bmatrix}
. \\ . \\ 1
\end{bmatrix}
\right),
\end{displaymath}
thus $p = x^2 -2 = (x-\sqrt{2})(x+\sqrt{2})$ in $\freeALG{\numR}{X}$.
\end{Example}

\begin{remark}
Note that it is easy to check that 
$p = xy -2 \in \freeALG{\field{K}}{X}$ 
is atomic, because both entries, $\alpha_{1,2}$ in $P$
and $\beta_{2,3}$ in $Q$ have to be zero
(otherwise the upper right entry in $A'$ could not
become zero) and hence there is \emph{no} non-trivial pre-standard
admissible transformation, that is, transformation that changes
the upper right block structure.
\end{remark}

\begin{lemma}\label{lem:pf.factorization}
Let $0\ne p,q_1,q_2 \in \freeALG{\field{K}}{X}$ be given by the \emph{minimal}
pre-standard admissible linear systems $\als{A} = (1,A,\lambda)$,
$\als{A}_1 = (1,A_1,\lambda_1)$ and $\als{A}_2 = (1,A_2,\lambda_2)$
of dimension $n=\rank p$ and $n_1,n_2 \ge 2$ respectively
with $p = q_1 q_2$.
Then there exists a pre-standard admissible transformation $(P,Q)$ such
that $PAQ$ has an upper right block of zeros of size
$(n_1-1) \times (n_2-1) = (n-n_2) \times (n-n_1)$.
\end{lemma}
\begin{proof}
Let $\als{A}'=(1,A',\lambda_2)$ be the minimal ALS from
Proposition~\ref{pro:pf.minmul}
for $q_1 q_2$. Clearly, we have $\dim \als{A}' =  n_1 + n_2 - 1 = n = \rank p$.
And $A'$ has, by construction, an upper right block of zeros of size
$(n_1-1) \times (n_2-1)$. Both, $\als{A}$ and $\als{A}'$ represent the
same element $p$, thus 
---by \cite[Theorem~1.4]{Cohn1999a}
--- there exists an admissible transformation $(P,Q)$
such that $P\als{A}Q = \als{A}'$. Since $\als{A}'$ is
pre-standard, $(P,Q)$ is a pre-standard admissible transformation.
\end{proof}

\begin{theorem}[Polynomial Factorization]\label{thr:pf.factorization}
Let $p \in \freeALG{\field{K}}{X}$ be given by the
\emph{minimal} pre-standard admissible linear system $\als{A} = (1,A,\lambda)$
of dimension $n = \rank p \ge 3$.
Then $p$ has a factorization into $p = q_1 q_2$ with $\rank(q_i) = n_i \ge 2$
if and only if there exists a pre-standard admissible transformation $(P,Q)$
such that $PAQ$ has an upper right block of zeros of size $(n_1 - 1) \times (n_2 - 1)$.
\end{theorem}
\begin{proof}
If there is such a factorization, then Lemma~\ref{lem:pf.factorization}
applies.
Conversely, if we have such a (pre-standard admissible) transformation
for a zero block of size $k_1 \times k_2$
we obtain an ALS in \emph{block} form ($p$ is the first entry of $s_{\block{1}}$)
\begin{displaymath}
\begin{bmatrix}
A_{1,1} & A_{1,2} & . \\
. & 1 & A_{2,3} \\ 
. & . & A_{3,3}
\end{bmatrix}
s =
\begin{bmatrix}
. \\ . \\ v_{\block{3}}
\end{bmatrix},
\quad
s = 
\begin{bmatrix}
s_{\block{1}} \\ g \\ s_{\block{3}}
\end{bmatrix},
\end{displaymath}
with
$A_{1,1}$ of size $k_1 \times k_1$
and $A_{3,3}$ of size $k_2 \times k_2$, 
in which we duplicate the entry $s_{k_1+1}$ by inserting
a ``dummy'' row (and column) to get the following ALS of
size $k_1 +k_2 +2 = n + 1$:
\begin{displaymath}
\begin{bmatrix}
A_{1,1} & A_{1,2} & 0  & . \\
0 & 1 & -1 & 0 \\
. & . & 1 & A_{2,3} \\ 
. & . & 0 & A_{3,3}
\end{bmatrix}
s' =
\begin{bmatrix}
. \\ . \\ . \\ v_{\block{3}}
\end{bmatrix},
\quad
s' = 
\begin{bmatrix}
s_{\block{1}} \\ g \\ g \\ s_{\block{3}}
\end{bmatrix}.
\end{displaymath}
According to the construction of the multiplication
in Proposition~\ref{pro:pf.ratop} we have $p = f g$
in the first component of $s_{\block{1}}$,
the first block in $s'$, for $f,g \in \freeALG{\field{K}}{X}$
given by the (pre-standard) admissible linear systems
\begin{displaymath}
\begin{bmatrix}
A_{1,1} & A_{1,2} \\
. & 1
\end{bmatrix}
s_f = 
\begin{bmatrix}
. \\ 1
\end{bmatrix}
\quad\text{and}\quad
\begin{bmatrix}
1 & A_{2,3} \\
. & A_{3,3} 
\end{bmatrix}
s_g =
\begin{bmatrix}
. \\ v_{\block{3}}
\end{bmatrix}
\end{displaymath}
of dimension $n_1 = k_1 + 1$ and $n_2 = k_2+1$ respectively.
\end{proof}

This finishes the algorithm. A simple analogue of 
\cite[Theorem~4.1]{Cohn1999a}
\ will do the rest, we do not have to worry about
invertibility of the transformation matrices $P$ and $Q$
in~\eqref{eqn:pf.factrn}.
In our case
$\field{K}[\alpha,\beta] = \field{K}[
\alpha_{1,2},\ldots,\alpha_{1,n-1},\alpha_{2,3},\ldots,\alpha_{2,n-1},\ldots,\alpha_{n-2,n-1},$
$\beta_{2,3},\ldots,\beta_{2,n},$
$\beta_{3,4},\ldots,\beta_{3,n},\ldots,\beta_{n-1,n} ]$.
However, a non-trivial ideal does \emph{not} guarantee
a solution over $\field{K}$ although
there are solutions over $\aclo{\field{K}}$.
To test only (without computing it)
\emph{if} there is a solution (over $\field{K}$)
one can use the concept of \emph{resultants}.
An introduction can be found in
\cite[Section~3.6]{Cox2015a}
. This book contains also an introduction
to Gröbner bases and an overview about
computer algebra software for computing them.
Additional to the work of 
\cite{Buchberger1970a}
, the survey on Gröbner--Shirshov bases
of \cite{Bokut2000a}
\ could be consulted.

\begin{remark}
Note, that in general in order to reverse the multiplication
(Proposition~\ref{pro:pf.minmul})
to find factors, we also need a lower left block of zeros of appropriate size.
This important fact is hidden in the pre-standard form of an ALS.
A general factorization concept is considered in future work.
\end{remark}

\begin{proposition}\label{pro:pf.ideal}
Let $p\in \freeALG{\field{K}}{X}$ be given by the
minimal pre-standard admissible linear system $\als{A} = (1,A,\lambda)$
of dimension $n = \rank p \ge 3$ and let $(P,Q)$ as in~\eqref{eqn:pf.factrn}.
Fix a $k \in \{ 1, 2, \ldots, n-2 \}$
and denote by $I_k$ the ideal
of $\field{K}[\alpha,\beta]$
generated by the coefficients of each $x\in \{ 1 \} \cup X$ in the
$(i,j)$ entries of the matrix $PAQ$ for $1 \le i \le k$ and
$k+2 \le j \le n$. Then $p$ factorizes
\emph{over} $\freeALG{\aclo{\field{K}}}{X}$
into $p=q_1 q_2$ with $\rank q_1 = k+1$ and $\rank q_2 = n-k$
if and only if the ideal $I_k$ is non-trivial.
\end{proposition}

Given a polynomial $p\in \freeALG{\field{K}}{X}$
by a minimal pre-standard ALS of dimension $n=\rank p \ge 2$
there are at most $\phi(n) = 2^{n-2}$ (minimal) standard admissible linear
systems (with respect to the structure of the upper right blocks
of zeros). For $n=2$ this is clear. For $n > 2$ the ALS could
be atomic or have a block of zeros of size $1 \times (n-2)$
or $(n-2) \times 1$, thus $\phi(n+1) = 1 + 2 \phi(n)- 1 = 2 \phi(n)$
because the system with ``finest'' upper right structure is
counted twice.

\begin{remark}
Recall that ---up to similarity--- each element $p\in \freeALG{\field{K}}{X}$
has only \emph{one} factorization into atoms.
If one is interested in the number of factorizations
(not necessarily irreducible)
up to permutation (and multiplication of the factors by units)
the above estimate $2^{\rank p-2}$ can be used.
However, the number of factorizations in the sense
of \cite[Definition~3.1]{Bell2017a}
\ can be bigger.
As an example take the polynomial $p = (x-1)(x-2)(x-3)$
which has $3! = 6$ different factorizations while
the number of standard admissible
linear systems is bounded by
$2^{\rank p-2} = 4$.
\end{remark}

\medskip

Let $p$ be a non-zero polynomial with
the factorization $p=q_1 q_2 \cdots q_m$
into atoms $q_i \in \freeALG{\field{K}}{X}$.
Since $\freeALG{\field{K}}{X}$ is a similarity-UFD
(Proposition~\ref{pro:pf.cohn63b}), \emph{each}
factorization of (a non-zero non-unit) $p$ into atoms has $m$ factors.
Therefore one can define the 
\emph{length}\index{length of a polynomial}
of $p$ by $m$, written as $\ell(p) = m$.
For a word $w \in X^* \subseteq \freeALG{\field{K}}{X}$
the length is $\ell(w) = \length{w}$.
By looking at the minimal polynomial multiplication
(Proposition~\ref{pro:pf.minmul})
it is easy to see that the length of an element
$p\in \freeALG{\field{K}}{X}^\bullet$
can be estimated by the rank, namely $\ell(p) \le \rank(p) - 1$.
More on length functions ---and transfer homomorphisms in the
context of non-unique factorizations--- (in the non-commutative setting)
can be found in
\cite[Section~3]{Smertnig2015a}
\ or
\cite{Baeth2015a}
.

\newpage
\section{Generalizing the Companion Matrix}\label{sec:pf.ncm}

For a special case, namely an alphabet with just one letter,
the companion matrix of a polynomial $p(x)$ yields immediately a
\emph{minimal} linear representation of $p \in \freeALG{\field{K}}{\{x\}}$.
If this is the characteristic polynomial of some (square) matrix
$B \in \field{K}^{m \times m}$,
then its eigenvalues can be computed by the techniques from
Section~\ref{sec:pf.fnp} (if necessary going over to $\aclo{\field{K}}$),
illustrated in Example~\ref{ex:pf.eig}.

For a broader class of (nc) polynomials,
\emph{left} and \emph{right} companion systems can be defined.
In general \emph{minimal} pre-standard admissible linear systems
are necessary to generalize companion matrices,
see Definition~\ref{def:pf.cm}.

\begin{definition}[Companion Matrix, Characteristic Polynomial, Normal Form
\protect{\cite[Section~6.6]{Gantmacher1986a}
}]
Let $p(x) = a_0 + a_1 x + \ldots + a_{m-1} x^{m-1} + x^m \in \field{K}[x]$.
The \emph{companion matrix} $L(p)$ is defined as
\begin{displaymath}
L(p) =
\begin{bmatrix}
0 & 0 & \ldots & 0 & -a_0 \\
1 & 0 & \ddots & \vdots & -a_1 \\
  & \ddots & \ddots & 0 & \vdots \\
  &  & 1 & 0 & -a_{m-2} \\
  &  &   & 1 & -a_{m-1}
\end{bmatrix}.
\end{displaymath}
Then $p(x)$ is the \emph{characteristic polynomial}\index{characteristic polynomial}
of $L = L(p)$:
\begin{displaymath}
\det(xI - L) = \det
\begin{bmatrix}
x & 0 & \ldots & 0 & a_0 \\
-1 & x & \ddots & \vdots & a_1 \\
  & \ddots & \ddots & 0 & \vdots \\
  &  & -1 & x & a_{m-2} \\
  &  &   & -1 & x+a_{m-1}
\end{bmatrix}.
\end{displaymath}
Given a square matrix $M \in \field{K}^{m \times m}$,
the \emph{normal form}%
\index{normal form (of a matrix)}
for $M$ can be defined in terms of the
\emph{companion matrix} $L(M)$ of its 
\emph{characteristic polynomial} $p(M) = \det (xI - M)$.
\end{definition}

\begin{remark}
In \cite[Section~8.1]{Cohn1995a}
, $C(p) = xI - L(p)\trp$ is also called \emph{companion matrix}.
This is justified by viewing $C(p)$ as \emph{linear matrix pencil}
$C(p) = C_0 \otimes 1 + C_x \otimes x$.
It generalizes nicely for non-commutative polynomials.
\end{remark}

Now we leave $\field{K}[x] = \freeALG{\field{K}}{\{x\}}$
and consider (nc) polynomials $p \in \freeALG{\field{K}}{X}$.
There are two cases where a \emph{minimal} ALS can be
stated immediately, namely if the support (of the polynomial)
can be ``built'' from the left (in the left family) or
from the right (in the right family)
with \emph{strictly} increasing rank.
For example,
a \emph{minimal} pre-standard ALS for
$p = a_0  + a_1 (x+2y) + a_2 (x-z)(x+2y) + y(x-z)(x+2y)$
is given by
\begin{displaymath}
\begin{bmatrix}
1 & -y -a_2 & -a_1 & -a_0 \\
. & 1 & -(x-z) & . \\
. & . & 1 & -(x+2y) \\
. & . & . & 1
\end{bmatrix}
s =
\begin{bmatrix}
. \\ . \\ . \\ 1
\end{bmatrix},
\quad s =
\begin{bmatrix}
p \\
(x-z)(x+2y) \\
x+2y \\
1
\end{bmatrix}.
\end{displaymath}

\begin{definition}[Left and Right Companion System]\label{def:pf.cs}
For $i=1,2,\ldots, m$ let $q_i \in \freeALG{\field{K}}{X}$
with $\rank{q_i} = 2$ and $a_i \in \field{K}$.
For a polynomial $p \in \freeALG{\field{K}}{X}$
of the form
\begin{displaymath}
p = q_m q_{m-1} \cdots q_1 + a_{m-1} q_{m-1} \cdots q_1 + \ldots + a_2 q_2 q_1
    + a_1 q_1 + a_0
\end{displaymath}
the pre-standard ALS
\begin{equation}\label{eqn:pf.lcs}
\begin{bmatrix}
1 & -q_m -a_{m-1} & -a_{m-2} & \ldots & -a_1 & -a_0 \\
  & 1 & -q_{m-1} & 0 & \ldots & 0 \\
  &   & \ddots & \ddots & \ddots & \vdots \\
  &   &   &  1 & -q_2 & 0 \\
  &   &   &  & 1 & -q_1 \\
  &   &   &  & & 1
\end{bmatrix}
s = 
\begin{bmatrix}
0 \\ 0 \\ \vdots \\ 0 \\ 0 \\ 1
\end{bmatrix}
\end{equation}
is called \emph{left companion system}\index{left companion system}.
And for a polynomial $p \in \freeALG{\field{K}}{X}$ of the form
\begin{displaymath}
p = a_0 + a_1 q_1 + a_2 q_1 q_2 + \ldots + a_{m-1} q_1 q_2 \cdots q_{m-1}
    + q_1 q_2 \cdots q_m
\end{displaymath}
the pre-standard ALS
\begin{equation}\label{eqn:pf.rcs}
\begin{bmatrix}
1 & -q_1 & 0  & \ldots & 0 & -a_0 \\
  & 1 & -q_2 & \ddots & \vdots & -a_1 \\
  &   & \ddots & \ddots & 0 & \vdots \\
  &   &   &  1 & -q_{m-1} & -a_{m-2} \\
  &   &   &  & 1 & -q_m -a_{m-1} \\
  &   &   &  & & 1
\end{bmatrix}
s = 
\begin{bmatrix}
0 \\ 0 \\ \vdots \\ 0 \\ 0 \\ 1
\end{bmatrix}
\end{equation}
is called \emph{right companion system}\index{right companion system}.
\end{definition}

\begin{proposition}\label{pro:pf.minsys}
For $i=1,2,\ldots, m$ let $q_i \in \freeALG{\field{K}}{X}$
with $\rank{q_i} = 2$. Then the polynomials
$p_{\text{l}} = q_m q_{m-1} \cdots q_1 + a_{m-1} q_{m-1} \cdots q_1
      + \ldots + a_2 q_2 q_1 + a_1 q_1 + a_0$
and
$p_{\text{r}} = a_0 + a_1 q_1 + a_2 q_1 q_2 
      + \ldots + a_{m-1} q_1 q_2 \cdots q_{m-1} + q_1 q_2 \cdots q_m$
have rank $m+1$.
\end{proposition}
\begin{proof}
Both, the left family $( p, q_{m-1} \cdots q_1, \ldots, q_2 q_1 , q_1, 1 )$
and the right family $( 1, q_m + a_{m-1},$ $(q_m+a_{m-1})q_{m-1} + a_{m-2}, \ldots, p)$
for \eqref{eqn:pf.lcs} are $\field{K}$-linearly independent.
Thus, the left companion system (for $p_{\text{l}}$) is minimal
of dimension $m+1$. Hence $\rank p_{\text{l}} = m+1$.
By a similar argument for the right companion system we
get $\rank p_{\text{r}} = m+1$.
\end{proof}

For a general polynomial $p\in \freeALG{\field{K}}{X}$ with $\rank p = n \ge 2$
we can take any \emph{minimal} pre-standard ALS $\als{A} = (1,A,\lambda)$
to obtain an ALS of the form $(1,A',1)$ by 
dividing the last row by $\lambda$ and multiplying the last column by $\lambda$.
Now we can define a (generalized version of the) companion matrix.
Those companion matrices can be used as building blocks to get
companion matrices for products of polynomials.
This is just a different point of view on the ``minimal'' polynomial
multiplication from Proposition~\ref{pro:pf.minmul}.

\begin{remark}
Although nothing can be said in general about minimality
of a linear representation for \emph{commutative} polynomials
(in several variables), Proposition~\ref{pro:pf.minsys}
can be used for constructing minimal linear representations
in the commutative case. Because in this case, the rank is
the maximum of the ranks of the monomials.
For example $p = x^2 y + xyz = xyx + xyz = xy(x+z)$.
\end{remark}

\begin{definition}[Companion Matrices]\label{def:pf.cm}
Let $p \in \freeALG{\field{K}}{X}$ with $\rank p = n \ge 2$ be given
by a \emph{minimal} pre-standard admissible linear system
$\als{A} = (1,A,1)$ and denote by $C(p)$ the upper right
submatrix of size $(n-1) \times (n-1)$.
Then $C(p)$ is called a (nc) \emph{companion matrix}\index{companion matrix}
of $p$.
\end{definition}

\begin{Example}\label{ex:pf.eig}
Let 
\begin{displaymath}
B =
\begin{bmatrix}
6 & 1 & 3 \\
-7 & 3 & 14 \\
1 & 0 & 1
\end{bmatrix}.
\end{displaymath}
Then the characteristic polynomial of $B$ is
$p(x) = \det(xI - B) = x^3 - 10 x^2 + 31 x - 30$.
The left companion system of $p$ is
\begin{displaymath}
\begin{bmatrix}
1 & -x+10 & -31 & 30 \\
. & 1 & -x & . \\
. & . & 1 & -x \\
. & . & . & 1
\end{bmatrix}
s =
\begin{bmatrix}
. \\ . \\ . \\ 1
\end{bmatrix},
\quad
s =
\begin{bmatrix}
p(x) \\
x^2 \\
x \\
1
\end{bmatrix}.
\end{displaymath}
Applying the transformation $(P,Q)$ with
\begin{displaymath}
P = 
\begin{bmatrix}
1 & . & . & . \\
  & 1 & -3 & . \\
  &   &  1 & . \\
  &   &    & 1
\end{bmatrix}
\begin{bmatrix}
1 & -5 & 6 & . \\
  & 1 & . & . \\
  &   &  1 & . \\
  &   &    & 1
\end{bmatrix}
=
\begin{bmatrix}
1 & -5 & 6 & . \\
  & 1 & -3 & . \\
  &   &  1 & . \\
  &   &    & 1
\end{bmatrix}
\end{displaymath}
and
\begin{displaymath}
Q = 
\begin{bmatrix}
1 & . & . & . \\
  & 1 & 5  & . \\
  &   &  1 & \frac{6}{5} \\
  &   &    & 1
\end{bmatrix}
\begin{bmatrix}
1 & . & . & . \\
  & 1 & . & . \\
  &   &  1 & \frac{9}{5} \\
  &   &    & 1
\end{bmatrix}
=
\begin{bmatrix}
1 & . & . & . \\
  & 1 & 5 & 9 \\
  &   &  1 & 3 \\
  &   &    & 1
\end{bmatrix}
\end{displaymath}
we get the ALS
\begin{displaymath}
\begin{bmatrix}
1 & 5-x & . & . \\
. & 1 & 2-x & . \\
. & . & 1 & 3-x \\
. & . & . & 1
\end{bmatrix}
s =
\begin{bmatrix}
. \\ . \\ . \\ 1
\end{bmatrix},
\quad
s =
\begin{bmatrix}
p(x) \\
(x-2)(x-3) \\
x-3 \\
1
\end{bmatrix}.
\end{displaymath}
Thus the eigenvalues of $B$ are $2$, $3$ and $5$.
Compare with Example~\ref{ex:pf.irred} for the case when
the polynomial does not decompose in linear factors.
\end{Example}

\newpage
\section{An Example (step by step)}\label{sec:pf.ex}

Let $p = x(1-yx)(3-yx)$ and $q = (xy-1)(xy-3)x$ be given.
Taking companion systems (Definition~\ref{def:pf.cs}) for
their factors respectively,
by Proposition~\ref{pro:pf.minmul}
we get the \emph{minimal} ALS (for $p$):
\begin{equation}\label{eqn:pf.ex.1}
\begin{bmatrix}
1 & - x & . & . & . & . \\
. & 1 & -y & -1 & . & . \\
. & . & 1 & x & . & . \\
. & . & . & 1 & -y & -3 \\
. & . & . & . & 1 & x \\
. & . & . & . & . & 1
\end{bmatrix}
s =
\begin{bmatrix}
. \\ . \\ . \\ . \\ . \\ 1
\end{bmatrix},
\quad
s =
\begin{bmatrix}
x(1-yx)(3-yx) \\ (1-yx)(3-yx) \\
-x(3-yx) \\ 3-yx \\ -x \\ 1
\end{bmatrix}
\end{equation}
Clearly, $p = xyxyx - 4xyx + 3x = q$
and $\rank p = \rank q = 6$.

\subsection{Constructing a minimal ALS}

If a minimal ALS (for some polynomial) cannot be stated directly by
a left or right companion system (Definition~\ref{def:pf.cs}),
an ALS for $p$
can be constructed using rational operations from
Proposition~\ref{pro:pf.ratop} (except the inverse)
out of monomials (or polynomials).
If the involved systems are pre-standard,
then the resulting system for the sum can be easily transformed
into a pre-standard ALS by row operations
and we can apply Algorithm~\ref{alg:pf.minals} to
obtain a minimal one.

Here we describe an alternative approach which works
directly on the left and right families like
\cite{Cardon1980a}
. It preserves the upper triangular form of the system
matrix. However, for an alphabet with more than one letter,
it is far from optimal because then the
dimension of the vector spaces (see below) might
grow exponentially. We leave it to the reader to
use the ideas of \cite{Cardon1980a}
\ to improve that significantely.

Recall that an ALS $\als{A} = (u,A,v)$ is minimal if and only if both the
left family $s = A\inv v$ is $\field{K}$-linearly independent and the
right family $t = u A\inv$ is $\field{K}$-linearly independent
(Proposition~\ref{pro:pf.cohn94.47}).
Let $f = 3x - 4xyx$. An ALS for $f$ of dimension $n=6$ is given by
\begin{displaymath}
\begin{bmatrix}
1 & - x & . & . & -1 & . \\
. & 1 & - y & . & . & . \\
. & . & 1 & -x & . & . \\
. & . & . & 1 & . & . \\
. & . & . & . & 1 & -x \\
. & . & . & . & . & 1
\end{bmatrix}
s =
\begin{bmatrix}
. \\ . \\ . \\ -4 \\ . \\ 3
\end{bmatrix},
\quad s =
\begin{bmatrix}
3x - 4xyx \\ -4yx \\ -4x \\ -4 \\ 3x \\ 3
\end{bmatrix}.
\end{displaymath}
Note that the free associative algebra $\freeALG{\field{K}}{X}$
is a \emph{vector space} with basis $X^*$,
the free monoid of the alphabet $X$.
Here $X^* = \{ 1, x, y, xx, xy, yx, yy, \ldots \}$.
Therefore, 
we can write the left family $s$ as a matrix (of coordinate row vectors)
$S \in \field{K}^{n \times m_s}$
with column indices $\{ 1, x, yx, xyx \}$
and the right family $t$ as a matrix (of coordinate column vectors)
$T \in \field{K}^{m_t \times n}$
with row indices $\{ 1, x, xy, xyx \}$, that is,
\begin{displaymath}
S = 
\begin{bmatrix}
. & 3 & . & -1 \\
. & . & -4 & . \\
. & -4 & . & . \\
-4 & . & . & . \\
. & 3 & . & . \\
3 & . & . & .
\end{bmatrix}
\quad\text{and}\quad
T = 
\begin{bmatrix}
1 & . & . & . & 1 & . \\
. & 1 & . & . & . & 1 \\
. & . & 1 & . & . & . \\
. & . & . & 1 & . & .
\end{bmatrix}.
\end{displaymath}
Both, $S$ and $T$, have rank~$4 < n = 6$, hence $\als{A}$ cannot be
minimal. The entries of the left (respectively right) family are just rows in $S$,
also denoted by $s_i$ for $i =1,\ldots,n$
(respectively columns in $T$, denoted by $t_j$ for $j=1,\ldots,n$).

Adding a row $s_k$ to $s_i$ results in \emph{subtracting} column~$i$
from column~$k$ in the system matrix $A$. In order to keep the
triangular structure $i < k$ must hold. Similarly,
adding a column $t_k$ to $t_j$ results in \emph{subtracting}
row~$j$ to row~$k$, hence $j > k$ must hold.
If we want to construct zeros in row~$i=3$ (we cannot produce zeros in
rows~1 and~2) in $S$ we need to find
an (invertible) transformation matrix $Q\in \field{K}^{n \times n}$
of the form (note that we will apply $Q\inv$ from the left)
\begin{displaymath}
Q =
\begin{bmatrix}
1 & . & . & . & . & . \\
  & 1 & . & . & . & . \\ 
  &   & 1 & \beta_4 & \beta_5 & \beta_6 \\
  &   &   & 1 & . & . \\
  &   &   &   & 1 & . \\
  &   &   &   &   & 1
\end{bmatrix}
\end{displaymath}
by solving the (underdetermined) linear system with
the system matrix consisting of the rows $i+1,i+2,\ldots,n$
and the right hand side consisting of the row~$i$
of~$S$:
\begin{displaymath}
\begin{bmatrix}
\beta_4 & \beta_5 & \beta_6
\end{bmatrix}
\begin{bmatrix}
-4 & . & . & . \\
. & 3 & . & . \\
3 & . & . & .
\end{bmatrix}
=
\begin{bmatrix}
. & -4 & . & . \\
\end{bmatrix}.
\end{displaymath}
Here $\beta_5 = -4/3$ and we choose $\beta_4 = \beta_6 = 0$.
Now row~3 in $Q\inv S$ is zero
\begin{displaymath}
Q\inv S =
\begin{bmatrix}
1 & . & . & . & . & . \\
  & 1 & . & . & . & . \\ 
  &   & 1 & 0 & 4/3 & 0 \\
  &   &   & 1 & . & . \\
  &   &   &   & 1 & . \\
  &   &   &   &   & 1
\end{bmatrix}
\begin{bmatrix}
. & 3 & . & -1 \\
. & . & -4 & . \\
. & -4 & . & . \\
-4 & . & . & . \\
. & 3 & . & . \\
3 & . & . & .
\end{bmatrix}
=
\begin{bmatrix}
. & 3 & . & -1 \\
. & . & -4 & . \\
. & 0 & . & . \\
-4 & . & . & . \\
. & 3 & . & . \\
3 & . & . & .
\end{bmatrix},
\end{displaymath}
that is $s'_3 = (Q\inv s)_3 = 0$. Therefore we can remove
row~3 and column~3 from the (admissibly) modified system
$(uQ, AQ, v)$ and get one of dimension $n-1=5$:
\begin{displaymath}
\begin{bmatrix}
1 & - x & . & -1 & . \\
. & 1 & . & \frac{4}{3} y & . \\
. & . & 1 & . & . \\
. & . & . & 1 & -x \\
. & . & . & . & 1
\end{bmatrix}
s =
\begin{bmatrix}
. \\ . \\ -4 \\ . \\ 3
\end{bmatrix},
\quad s =
\begin{bmatrix}
3x - 4xyx \\ -4yx \\ -4 \\ 3x \\ 3
\end{bmatrix}.
\end{displaymath}
The right family of the new ALS is $\bigl(1,x,0,1-\frac{4}{3} xy, x-\frac{4}{3} xyx\bigr)$.
After adding $4/3$-times row~5 to row~3 we can remove row~3 and column~3
and get a \emph{minimal} ALS for $3x - 4xyx = (1-\frac{4}{3} xy) 3x$:
\begin{displaymath}
\begin{bmatrix}
1 & -x & -1 & . \\
. & 1 & \frac{4}{3} y & . \\
. & . & 1 & -x \\
. & . & . & 1
\end{bmatrix}
s =
\begin{bmatrix}
. \\ . \\ . \\ 3
\end{bmatrix},
\quad
s = 
\begin{bmatrix}
3x - 4xyx \\
-4yx \\
3x \\
3
\end{bmatrix}.
\end{displaymath}

\subsection{Factorizing a Polynomial}

Now we consider the following \emph{minimal} ALS for $p = xyxyx + (3x - 4xyx)$,
constructed in a similar way as shown in the previous subsection 
or using the right companion system (Definition~\ref{def:pf.cs}):
\begin{displaymath}
\begin{bmatrix}
1 & - x & . & . & . & -x \\
. & 1 & -y & . & . & . \\
. & . & 1 & -x & . & \frac{4}{3} x \\
. & . & . & 1 & -y & . \\
. & . & . & . & 1 & -\frac{1}{3} x \\
. & . & . & . & . & 1
\end{bmatrix}
s =
\begin{bmatrix}
. \\ . \\ . \\ . \\ . \\ 3
\end{bmatrix}.
\end{displaymath}
We try to create an upper right block of zeros of size $3 \times 2$.
For that we apply the (admissible) transformation $(P,Q)$ directly
to the coefficient matrices $A_0$, $A_x$ and $A_y$ in
$A = A_0 \otimes 1 + A_x \otimes x + A_y \otimes y$
to get the equations. For $y$ we have
\begin{align*}
P A_y Q &=
P
\begin{bmatrix}
. & . & . & . & . & . \\
  & . & -1 & . & . & . \\
  &   & . & . & . & . \\
  &   &   & . & -1 & . \\
  &   &   &   & . & . \\
  &   &   &   &   & . 
\end{bmatrix}
\begin{bmatrix}
1 & 0 & 0 & 0 & 0 & 0 \\
  & 1 & \beta_{2,3} & \beta_{2,4} & \beta_{2,5} & \beta_{2,6} \\
  &   & 1 & \beta_{3,4} & \beta_{3,5} & \beta_{3,6} \\
  &   &   & 1 & \beta_{4,5} & \beta_{4,6} \\
  &   &   &   & 1 & \beta_{5,6} \\
  &   &   &   &   & 1
\end{bmatrix}
\end{align*}
\begin{align*}
\quad\quad\quad
&=
\begin{bmatrix}
1 & \alpha_{1,2} & \alpha_{1,3} & \alpha_{1,4} & \alpha_{1,5} & 0 \\
  & 1 & \alpha_{2,3} & \alpha_{2,4} & \alpha_{2,5} & 0 \\
  &   & 1 & \alpha_{3,4} & \alpha_{3,5} & 0 \\
  &   &   & 1 & \alpha_{4,5} & 0 \\
  &   &   &   & 1 & 0 \\
  &   &   &   &   & 1
\end{bmatrix}
\begin{bmatrix}
. & . & . & . & . & . \\
  & . & -1 & -\beta_{3,4} & -\beta_{3,5} & -\beta_{3,6} \\
  &   & . & . & . & . \\
  &   &   & . & -1 & -\beta_{5,6} \\
  &   &   &   & . & . \\
  &   &   &   &   & . 
\end{bmatrix} \\
&= 
\begin{bmatrix}
. & . & -\alpha_{1,2} & -\alpha_{1,2}\beta_{3,4} &
      -\alpha_{1,2}\beta_{3,5} - \alpha_{1,4} &
      -\alpha_{1,4}\beta_{5,6} - \alpha_{1,2} \beta_{3,6} \\
  & . & -1 & -\beta_{3,4} & -\beta_{3,5} - \alpha_{2,4} &
      -\alpha_{2,4}\beta_{5,6} - \beta_{3,6} \\
  &   & . & . & -\alpha_{3,4} & -\alpha_{3,4}\beta_{5,6} \\
  &   &   & . & -1 & -\beta_{5,6} \\
  &   &   &   & . & . \\
  &   &   &   &   & . 
\end{bmatrix}
\end{align*}
of which we pick the upper right $3 \times 2$ block. Thus we
have the following~6 equations for $y$:
\begin{displaymath}
\begin{bmatrix}
\alpha_{1,2}\beta_{3,5} + \alpha_{1,4} &
      \alpha_{1,4}\beta_{5,6} + \alpha_{1,2} \beta_{3,6} \\
\beta_{3,5} + \alpha_{2,4} & 
      \alpha_{2,4}\beta_{5,6} + \beta_{3,6} \\
\alpha_{3,4} & \alpha_{3,4}\beta_{5,6}
\end{bmatrix}
=
\begin{bmatrix}
0 & 0 \\
0 & 0 \\
0 & 0 
\end{bmatrix}.
\end{displaymath}
Similarly, we get the following 6 equations for $x$:
\begin{displaymath}
\begin{bmatrix}
\alpha_{1,3} \beta_{4,5} + \beta_{2,5} &
  \alpha_{1,3} \beta_{4,6} + \beta_{2,6} + \frac{1}{3} \alpha_{1,5} - \frac{4}{3} \alpha_{1,3} + 1 \\
\alpha_{2,3} \beta_{4,5} &
  \alpha_{2,3} \beta_{4,6} + \frac{1}{3} \alpha_{2,5} - \frac{4}{3} \alpha_{2,3} \\
\beta_{4,5} & \beta_{4,6} + \frac{1}{3} a_{3,5} - \frac{4}{3} 
\end{bmatrix}
= 0.
\end{displaymath}
And finally those for $1$ (without terms containing
$\alpha_{3,4} = \beta_{4,5} = 0$):
\begin{displaymath}
\begin{bmatrix}
\alpha_{1,3} \beta_{3,5} + \alpha_{1,2}\beta_{2,5} + \alpha_{1,5} &
  \alpha_{1,5} \beta_{5,6} + \alpha_{1,4} \beta_{4,6} +
  \alpha_{1,3} \beta_{3,6} + \alpha_{1,2} \beta_{2,6} \\
\alpha_{2,3} \beta_{3,5} + \beta_{2,5} + \alpha_{2,5} &
  \alpha_{2,5} \beta_{5,6} + \alpha_{2,4} \beta_{4,6} +
  \alpha_{2,3} \beta_{3,6} + \beta_{2,6} \\
\beta_{3,5} + \alpha_{3,5} &
  \alpha_{3,5} \beta_{5,6} + \beta_{3,6}
\end{bmatrix}
= 0.
\end{displaymath}

\medskip
A Gröbner basis for the ideal generated by these 18 equations 
(computed by \textsc{FriCAS},
\cite{FRICAS2016}
\ using \emph{lexicographic order})
is $\bigl(
\alpha_{1,2} - \alpha_{1,4}\beta_{4,6},\;
\alpha_{1,3} - \alpha_{1,5} \beta_{4,6},\;
\alpha_{2,3} - \alpha_{2,5} \beta_{4,6},\;
\alpha_{2,4} + 3 \beta_{4,6} - 4,\;
\alpha_{3,4},\;
\alpha_{3,5} + 3\beta_{4,6} - 4,\;
\beta_{2,5},\;
\beta_{2,6} + 1,\;
\beta_{3,5} - 3 \beta_{4,6} + 4,\;
\beta_{3,6} - 3 \beta_{4,6} \beta_{5,6} + 4 \beta_{5,6},\;
\beta_{4,5},\;
\beta_{4,6}^2 - \tsfrac{4}{3} \beta_{4,6} + \tsfrac{1}{3} \bigr)$.
The last (over $\field{K}$ reducible)
generator is $(\beta_{4,6} -1)(\beta_{4,6}-\frac{1}{3})$,
thus $\beta_{4,6} \in \{ 1, \frac{1}{3} \}$.
We proceed with the case $\beta_{4,6}=1$.
We then have
$\alpha_{2,4} = \alpha_{3,5} = 1$,
$\beta_{2,6} = \beta_{3,5} = -1$
and choose
$\alpha_{1,2} = \alpha_{1,4} = 0$,
$\alpha_{1,3} = \alpha_{1,5} = 3$,
$\alpha_{2,3} = \alpha_{2,5} = 0$,
$\beta_{3,6} = -\beta_{5,6} = 0$.
Note that $\alpha_{4,5}$, $\beta_{2,3}$, $\beta_{2,4}$ and
$\beta_{3,4}$ are not present, so we set them to zero.
Therefore one possible transformation is
\begin{displaymath}
(P,Q) = \left(
\begin{bmatrix}
1 & 0 & 3 & 0 & 3 & . \\
  & 1 & 0 & 1 & 0 & . \\
  &   & 1 & 0 & 1 & . \\
  &   &   & 1 & 0 & . \\
  &   &   &   & 1 & . \\
  &   &   &   &   & 1
\end{bmatrix},
\begin{bmatrix}
1 & . & . & . & . & . \\
  & 1 & 0 & 0 & 0 & -1 \\
  &   & 1 & 0 & -1 & 0 \\
  &   &   & 1 & 0 & 1 \\
  &   &   &   & 1 & -0 \\
  &   &   &   &   & 1
\end{bmatrix}
\right)
\end{displaymath}
leading to the admissible linear system
$\als{A}' = P \als{A} Q$
\begin{displaymath}
\begin{bmatrix}
1 & -x & 3 & -3x & 0 & 0 \\
. & 1 & -y & 1 & 0 & 0 \\
. & . & 1 & -x & 0 & 0 \\
. & . & . & 1 & -y & 1 \\
. & . & . & . & 1 & -\tsfrac{1}{3}x \\
. & . & . & . & . & 1 
\end{bmatrix}
s =
\begin{bmatrix}
. \\ . \\ . \\ . \\ . \\ 3
\end{bmatrix},
\end{displaymath}
compare with the ALS \eqref{eqn:pf.ex.1}.
$\als{A}'$ is the (minimal) product of
\begin{equation}\label{eqn:pf.ex.3}
\begin{bmatrix}
1 & -x & 3 & -3x \\
. & 1 & -y & 1 \\
. & . & 1 & -x \\
. & . & . & 1 
\end{bmatrix}
s =
\begin{bmatrix}
. \\ . \\ . \\ 1 
\end{bmatrix},
\quad
s =
\begin{bmatrix}
xyx - x \\
yx - 1 \\
x \\
1
\end{bmatrix}
\end{equation}
and
\begin{equation}\label{eqn:pf.ex.4}
\begin{bmatrix}
 1 & -y & 1 \\
 . & 1 & -\tsfrac{1}{3}x \\
 . & . & 1 
\end{bmatrix}
s =
\begin{bmatrix}
 . \\ . \\ 3
\end{bmatrix},
\quad
s =
\begin{bmatrix}
yx - 3 \\
x \\
3
\end{bmatrix}.
\end{equation}
Thus $p = (xyx-x)(yx-3)$.
The first factor is \emph{not} atomic because we could (pre-standard
admissibly) construct
either an $1 \times 2$ upper right block of zeros
(by subtracting $3$-times row~3 from row~1)
or an $2 \times 1$ upper right block of zeros
(by subtracting column~2 from column~4
and subtracting $2$-times row~3 from row~1)
in the system matrix of $\eqref{eqn:pf.ex.3}$.
On the other hand,
a brief look at the ALS $\eqref{eqn:pf.ex.4}$
immediately shows that the second factor is irreducible.

\section*{Acknowledgement}

I thank Franz Lehner for the fruitful discussions
and his support, Daniel Smertnig who helped me
to find some orientation in abstract non-commutative factorization
(especially with literature),
Roland Speicher for the creative environment in Saarbrücken
and the anonymous referees for the valuable comments.

\ifJOURNAL
\else
\bibliographystyle{alpha}
\bibliography{doku}

\begin{thebibliography}{{Sme}15}

\bibitem[ARJ15]{Arvind2015b}
V.~Arvind, G.~Rattan, and P.~Joglekar.
\newblock On the complexity of noncommutative polynomial factorization.
\newblock In {\em Mathematical foundations of computer science 2015. {P}art
  {II}}, volume 9235 of {\em Lecture Notes in Comput. Sci.}, pages 38--49.
  Springer, Heidelberg, 2015.

\bibitem[BHL17]{Bell2017a}
J.~P. Bell, A.~Heinle, and V.~Levandovskyy.
\newblock On noncommutative finite factorization domains.
\newblock {\em Trans. Amer. Math. Soc.}, 369(4):2675--2695, 2017.

\bibitem[BK00]{Bokut2000a}
L.~A. Bokut$^\prime$ and P.~S. Kolesnikov.
\newblock Gr\"obner-{S}hirshov bases: from inception to the present time.
\newblock {\em Zap. Nauchn. Sem. S.-Peterburg. Otdel. Mat. Inst. Steklov.
  (POMI)}, 272(Vopr. Teor. Predst. Algebr i Grupp. 7):26--67, 345, 2000.

\bibitem[BR11]{Berstel2011a}
J.~Berstel and C.~Reutenauer.
\newblock {\em Noncommutative rational series with applications}, volume 137 of
  {\em Encyclopedia of Mathematics and its Applications}.
\newblock Cambridge University Press, Cambridge, 2011.

\bibitem[BS15]{Baeth2015a}
N.~R. Baeth and D.~Smertnig.
\newblock Factorization theory: from commutative to noncommutative settings.
\newblock {\em J. Algebra}, 441:475--551, 2015.

\bibitem[Buc70]{Buchberger1970a}
B.~Buchberger.
\newblock Ein algorithmisches {K}riterium f\"ur die {L}\"osbarkeit eines
  algebraischen {G}leichungssystems.
\newblock {\em Aequationes Math.}, 4:374--383, 1970.

\bibitem[{Car}10]{Caruso2010a}
F.~{Caruso}.
\newblock Factorization of non-commutative polynomials.
\newblock {\em ArXiv e-prints}, February 2010.

\bibitem[CC80]{Cardon1980a}
A.~Cardon and M.~Crochemore.
\newblock D\'etermination de la repr\'esentation standard d'une s\'erie
  reconnaissable.
\newblock {\em RAIRO Inform. Th\'eor.}, 14(4):371--379, 1980.

\bibitem[CLO15]{Cox2015a}
D.~A. Cox, J.~Little, and D.~O'Shea.
\newblock {\em Ideals, varieties, and algorithms}.
\newblock Undergraduate Texts in Mathematics. Springer, Cham, fourth edition,
  2015.
\newblock An introduction to computational algebraic geometry and commutative
  algebra.

\bibitem[Coh63]{Cohn1963b}
P.~M. Cohn.
\newblock Noncommutative unique factorization domains.
\newblock {\em Trans. Amer. Math. Soc.}, 109:313--331, 1963.

\bibitem[Coh72]{Cohn1972a}
P.~M. Cohn.
\newblock Generalized rational identities.
\newblock In {\em Ring theory ({P}roc. {C}onf., {P}ark {C}ity, {U}tah, 1971)},
  pages 107--115. Academic Press, New York, 1972.

\bibitem[Coh85]{Cohn1985a}
P.~M. Cohn.
\newblock {\em Free rings and their relations}, volume~19 of {\em London
  Mathematical Society Monographs}.
\newblock Academic Press, Inc. [Harcourt Brace Jovanovich, Publishers], London,
  second edition, 1985.

\bibitem[Coh95]{Cohn1995a}
P.~M. Cohn.
\newblock {\em Skew fields}, volume~57 of {\em Encyclopedia of Mathematics and
  its Applications}.
\newblock Cambridge University Press, Cambridge, 1995.
\newblock Theory of general division rings.

\bibitem[Coh03]{Cohn2003b}
P.~M. Cohn.
\newblock {\em Further algebra and applications}.
\newblock Springer-Verlag London, Ltd., London, 2003.

\bibitem[CR94]{Cohn1994a}
P.~M. Cohn and C.~Reutenauer.
\newblock A normal form in free fields.
\newblock {\em Canad. J. Math.}, 46(3):517--531, 1994.

\bibitem[CR99]{Cohn1999a}
P.~M. Cohn and C.~Reutenauer.
\newblock On the construction of the free field.
\newblock {\em Internat. J. Algebra Comput.}, 9(3-4):307--323, 1999.
\newblock Dedicated to the memory of Marcel-Paul Sch{\"u}tzenberger.

\bibitem[Dem97]{Demmel1997a}
J.~W. Demmel.
\newblock {\em Applied numerical linear algebra}.
\newblock Society for Industrial and Applied Mathematics (SIAM), Philadelphia,
  PA, 1997.

\bibitem[Fli74]{Fliess1974a}
M.~Fliess.
\newblock Matrices de {H}ankel.
\newblock {\em J. Math. Pures Appl. (9)}, 53:197--222, 1974.

\bibitem[FM80]{Fornasini1980a}
E.~Fornasini and G.~Marchesini.
\newblock On the problems of constructing minimal realizations for
  two-dimensional filters.
\newblock {\em IEEE Transactions on Pattern Analysis and Machine Intelligence},
  2(2):172--176, 1980.

\bibitem[Fri16]{FRICAS2016}
{\em \textsc{FriCAS} Computer Algebra System}, 2016.
\newblock W.~Hebisch,\\ \url{http://axiom-wiki.newsynthesis.org/FrontPage},\\
  \href{svn://svn.code.sf.net/p/fricas/code/trunk fricas}{\texttt{svn co
  svn://svn.code.sf.net/p/fricas/code/trunk fricas}}.

\bibitem[Gan86]{Gantmacher1986a}
F.~R. Gantmacher.
\newblock {\em Matrizentheorie}.
\newblock Springer-Verlag, Berlin, 1986.
\newblock With an appendix by V. B. Lidskij, With a preface by D. P.
  {\v{Z}}elobenko, Translated from the second Russian edition by Helmut Boseck,
  Dietmar Soyka and Klaus Stengert.

\bibitem[GHK06]{Geroldinger2006a}
A.~Geroldinger and F.~Halter-Koch.
\newblock {\em Non-unique factorizations}, volume 278 of {\em Pure and Applied
  Mathematics (Boca Raton)}.
\newblock Chapman \& Hall/CRC, Boca Raton, FL, 2006.
\newblock Algebraic, combinatorial and analytic theory.

\bibitem[GRW01]{Gelfand2001a}
I.~Gelfand, V.~Retakh, and R.~L. Wilson.
\newblock Quadratic linear algebras associated with factorizations of
  noncommutative polynomials and noncommutative differential polynomials.
\newblock {\em Selecta Math. (N.S.)}, 7(4):493--523, 2001.

\bibitem[HL13]{Heinle2013a}
A.~{Heinle} and V.~{Levandovskyy}.
\newblock Factorization of z-homogeneous polynomials in the first (q)-weyl
  algebra.
\newblock {\em ArXiv e-prints}, January 2013.

\bibitem[HV07]{Helton2007b}
J.~W. Helton and V.~Vinnikov.
\newblock Linear matrix inequality representation of sets.
\newblock {\em Comm. Pure Appl. Math.}, 60(5):654--674, 2007.

\bibitem[Jor89]{Jordan1989a}
D.~A. Jordan.
\newblock Unique factorisation of normal elements in noncommutative rings.
\newblock {\em Glasgow Math. J.}, 31(1):103--113, 1989.

\bibitem[LH18]{Levandovskyy2018a}
V.~Levandovskyy and A.~Heinle.
\newblock A factorization algorithm for {$G$}-algebras and its applications.
\newblock {\em J. Symbolic Comput.}, 85:188--205, 2018.

\bibitem[Ret10]{Retakh2010a}
V.~Retakh.
\newblock From factorizations of noncommutative polynomials to combinatorial
  topology.
\newblock {\em Cent. Eur. J. Math.}, 8(2):235--243, 2010.

\bibitem[{Sch}17]{Schrempf2017a}
K.~{Schrempf}.
\newblock Linearizing the {W}ord {P}roblem in (some) {F}ree {F}ields.
\newblock {\em ArXiv e-prints}, January 2017.

\bibitem[{Sme}15]{Smertnig2015a}
D.~{Smertnig}.
\newblock Factorizations of elements in noncommutative rings: A survey.
\newblock {\em ArXiv e-prints}, July 2015.

\bibitem[SS78]{Salomaa1978a}
A.~Salomaa and M.~Soittola.
\newblock {\em Automata-theoretic aspects of formal power series}.
\newblock Springer-Verlag, New York-Heidelberg, 1978.
\newblock Texts and Monographs in Computer Science.

\end{thebibliography}
\addcontentsline{toc}{section}{Bibliography}
\fi

\end{document}